\documentclass[12pt,epsfig,amsfonts]{amsart}
\pdfoutput=1
\usepackage{hyperref, soul}
\usepackage{graphicx}

\usepackage{euscript}
\usepackage{amsmath,amssymb,amscd,epsfig}
\usepackage{mathrsfs}
\usepackage{todonotes}

\usepackage{xcolor}
\newcommand{\rl}{\mathbb R}
\newcommand{\R}{\mathbb R}
\newcommand{\Int}{\text{Int}\,}
\newcommand{\T}{\mathbb T}
\newcommand{\inte}{\mathbb N}
\newcommand{\N}{\mathbb N}
\newcommand{\Px}{\mathcal{P}}

\newcommand{\Z}{\mathbb Z}

\newtheorem{rmk}{Remark}
\newtheorem{thm}{Theorem}[section]
\newtheorem{prop}[thm]{Proposition}
\newtheorem{lem}[thm]{Lemma}

\newcommand{\J}{J^u}

\begin{document}

\title[Thermodynamics of Towers of Hyperbolic Type]{Thermodynamics of Towers of Hyperbolic Type}

\author{Y. Pesin}
\address{Department of Mathematics, Pennsylvania State
University, University Park, PA 16802, USA}
\email{pesin@math.psu.edu}
\author{S. Senti}
\address{Instituto de Matematica, Universidade Federal do Rio de Janeiro, C.P. 68 530, CEP 21941-909, R.J., Brazil}
\email{senti@im.ufrj.br}
\author{K. Zhang}
\address{Department of Mathematics,
University of Toronto, Toronto, Ontario, Canada}
\email{kzhang@math.utoronto.ca}

%\date{\today}

\thanks{Y. Pesin is partially supported by the
National Science Foundation grant \#DMS-1101165. S. Senti acknowledges the support of the CNPq and CAPES} 

\begin{abstract} We introduce a class of continuous maps
$f$ of a compact topological space $X$ admitting inducing schemes of hyperbolic type and describe the associated tower constructions. We then establish a thermodynamic formalism, i.e., we describe a class of real-valued potential functions $\varphi$ on $X$ such that $f$ possess a unique equilibrium measure $\mu_\varphi$, associated to each $\varphi$, which minimizes the free energy among the measures that are liftable to the tower. We also describe some ergodic properties of equilibrium measures including decay of correlations and the central limit theorem. We then study the liftability problem and show that under some additional assumptions on the inducing scheme every measure that charges the base of the tower and has sufficiently large entropy is liftable. Our results extend those obtained in \cite{PesSen05, PesSen08} for inducing schemes of expanding types and apply to certain multidimensional maps. Applications include obtaining the thermodynamic formalism for Young's diffeomorphisms, the H\'enon family at the first birfucation and the Katok map. In particular for the Katok map we obtain the exponential decay of correlations for equilibrium measures associated to the geometric potentials with $0\le t<1$.
\end{abstract}

\maketitle

\section{Introduction}

For a continuous map $f$ of a compact metric space $X$ and a \emph{potential} function $\varphi$, an invariant Borel probability measure $\mu_\varphi$ is called an \emph{equilibrium measure} if
$$
h_{\mu_\varphi}(f)+\int_X\varphi d\mu_\varphi =\sup_{\mathcal{M}(f,X)}\{h_\mu(f)+\int_X\varphi d\mu\},
$$
where $\mathcal{M}(f,X)$ is the class of all $f$-invariant ergodic Borel probability measures. For maps admitting inducing schemes for some classes of potential functions (sometimes discontinuous), methods in \cite{PesSen08} prove the  existence of unique equilibrium measures within the class $\mathcal{M}_L(f,Y)$ of all \emph{liftable measure} (to be defined later) on an invariant subset $Y\subset X$. In some cases (for example, for some unimodal maps $f$ and the family of geometric potentials $\varphi_t=-t\log|f'|$), these equilibrium measures are shown to be the equilibrium measures in the classical sense.

The inducing schemes introduced in \cite{PesSen05, PesSen08} were designed to study non-uniformly expanding maps and can be thus called inducing schemes of \emph{expanding type}. In this paper we mainly consider maps which admit inducing schemes of \emph{hyperbolic type}. We shall show that this class of maps includes diffeomorphisms introduced by Young in \cite{You98}. They allow to construct a \emph{Young tower} which yields many interesting applications. 

While maps with inducing schemes of expanding type are modeled by towers over the full shift of countable type on the space of \emph{one-sided} infinite sequences, the maps with inducing schemes of hyperbolic type are modeled by towers over the full shift of countable type on the space of \emph{two-sided} infinite sequences. This affects our study in two directions. First, one needs to extend the  thermodynamic formalism for the full shift on the space of one-sided infinite sequences over a countable alphabet type to the full shift on the space of two-sided infinite sequences over a countable alphabet. This is done using the techniques due to Bowen \cite{Bow75} and Sarig \cite{Sar99}. 

Second, one need to solve the \emph{liftability problem} for towers of hyperbolic type, i.e., describe the class of measures that ``charge the tower'' (i.e., give positive weight to the base of the tower) and that can be lifted. In \cite{PesSenZha08}, we described a method to study the liftability problem for inducing schemes of expanding type; however, inducing schemes of hyperbolic type require a completely new approach that we develop here. Roughly speaking it states that if the number 
$S_n$ of the elements in the base of the tower with inducing time $n$ grows exponentially with an exponent $0\le h<h_{\text{top}}(f)$, then every measure $\mu$ whose metric entropy $h_\mu(f)>h$ is liftable. 

Reducing thermodynamic formalism to measures of large entropy is neither surprising nor artificial. Indeed, beyond the setting of the classical thermodynamics of uniformly hyperbolic maps and H\"older continuous potentials, measures of small (or even zero entropy) often appear as somewhat ``non-essential'' equilibrium measures that come in pairs with ``essential'' equilibrium measures of large entropy (see \cite{PesZha06}). The exponent $h$ in the exponential growth of the number $S(n)$ is a natural cut-off between the classes of essential and non-essential equilibrium measures. 

The present article is structured as follows. In Section 2 we introduce inducing schemes of hyperbolic type and their symbolic representation by a tower. In Section 3 we discuss thermodynamics of the full shift on the space of two-sided sequences over a countable alphabet. Our main result in this section is Theorem \ref{gibbs2} which provides conditions on the potential functions that guarantee existence and uniqueness of equilibrium measures and describes their ergodic properties. Theorem \ref{sarig1} establishes the exponential decay of correlations and the Central Limit Theorem (CLT) for equilibrium measures associated to locally H\"older continuous potentials. 

We use these results in Section 4 to establish the thermodynamic formalism for maps with inducing schemes of hyperbolic type. In particular, we present \emph{verifiable} conditions on potentials which guarantee existence and uniqueness of the associated equilibrium measures and allow the description of their ergodic properties (including exponential decay of correlations and CLT; see Theorems \ref{liftgibbs} and \ref{equilibrium1}).  

In Section 5 we solve the liftability problem for towers associated with inducing schemes of hyperbolic type. In Theorem~\ref{liftability} two additional conditions (L1) and (L2) are imposed on the inducing scheme to ensure that any invariant ergodic probability measure which charges the base of the tower and has sufficiently large entropy is liftable. 

In Sections 6 and 7 we apply our results to establish the thermodynamic formalism for Young's diffeomorphisms. In Proposition~\ref{Jacobian} we construct inducing schemes for these maps that satisfy all our requirements and prove the existence and uniqueness of equilibrium measures within the class of liftable measures for the geometric potentials 
$\varphi_t(x):=-t\log|\det(df|E^u(x))|$ with $t_0<t\le 1$ and some $t_0<0$ (here $E^u(x)$ denotes the unstable subspace at the point $x$, see Theorem \ref{geom_poten}). In particular, this proves uniqueness of measures of maximal entropy (within the class of liftable measures). We also show that for $t_0<t< 1$ the equilibrium measures have exponential decay of correlations and satisfy the Central Limit Theorem (see Theorem \ref{geom_poten1}). We comment that in general the equilibrium measure for $t=1$ may not have exponential decay of correlations. 

In Sections 8 and 9 we present two particular examples of maps -- the H\'enon family at the first bifurcation and the Katok map. The first example is studied in \cite{SenTak13} and the second one in \cite{PesSenZha14} and both are briefly described here. For these examples we show how to construct inducing schemes of hyperbolic type and  use our results to establish existence and uniqueness of equilibrium measures for the corresponding geometric potentials and we also show that these equilibrium measures have exponential decay of correlations and satisfy the Central Limit Theorem. In these particular eaxmples the inducing time is the first return time to the base of the tower and hence, every measure that charge the base is liftable.

{\bf Acknowledgments.} The authors wish to thank J. Buzzi for suggestions in dealing with the liftability of measures. 

Part of this work was done while the authors visited the Centre Interfacultaire Bernoulli (CIB), Ecole Polytechnique Federale de Lausanne, Switzerland. We would like to thank the Center for their hospitality. 

S. Senti would like to thank the support of Penn State Mathematics department and 
the Shapiro Visitor Program.

\section{Maps with inducing schemes}
\label{sec:ind-hyperbolic}

Let $f:X\to X$ be a continuous map of a compact metric space $(X, d)$. Throughout the paper we shall assume that $f$ has finite topological entropy $h_{top}(f)<\infty$.

Given a countable collection of disjoint Borel sets $S=\{J\}$ and  a positive integer-valued function $\tau: S\to\inte$, we say that  $f$ admits an \emph{inducing scheme of hyperbolic type} $\{S, \tau\}$, with \emph{inducing domain} 
$W:=\bigcup_{J\in S}J$ and \emph{inducing time} 
$\tau:X\to \inte$ defined by
$$
\tau(x):=
\begin{cases}
\tau(J), &  x\in J \\
0, &  x\notin  W
\end{cases}
$$
provided the following two conditions (I1)--(I2) hold:
\begin{enumerate}
\bigskip
\item[{\bf (I1)}] For any $J\in S$ we have
$$
f^{\tau(J)}(J)\subset W \quad\mbox{ and }\quad\bigcup_{J\in S}f^{\tau(J)}(J)=W.
$$
Moreover, $f^{\tau(J)}|J$ can be extended to a homeomorphism of a neighborhood of $J$;
\end{enumerate}
\bigskip
Condition (I1) allows one to define the \emph{induced map} 
$F\colon W\to W$ by 
$$
F|J:=f^{\tau(J)}|J, \qquad J\in S.
$$
For each $J\in S$, the map $F|J$ can be extended to the closure $\bar{J}$ by continuity. However, in general the extensions may not agree on points in $\bar{J}\cap \bar{J'}$ for some $J, J'\in S$ (see Proposition~\ref{conjugacy}).

If $f$ is invertible and $\tau$ is the first return time to $W$, then all images $f^{\tau(J)}(J)$ are disjoint. 
However, in general the sets $f^{\tau(J)}(J)$ corresponding to different $J\in S$ may overlap. In this case the map $F$ may not be invertible. 

\begin{enumerate}
\bigskip
\item[{\bf (I2)}] 
For every bi-infinite sequence $\underline{a}=(a_n)_{n\in\Z}\in S^{\Z}$ there exists a \emph{unique} sequence $\underline{x}=\underline{x}(\underline{a})=(x_n=x_n(\underline{a}))_{n\in\Z}$ such that
\begin{enumerate}
\item $x_n\in\overline{J_{a_n}}\quad\mbox{ and }\quad f^{\tau(J_{a_n})}(x_n)=x_{n+1}$;
\item if $x_n(\underline{a})=x_n(\underline{b})$ for all $n\le 0$ then 
$\underline{a}=\underline{b}$. 
\end{enumerate}
\end{enumerate}
\bigskip
Condition (I2) allows one to define the \emph{coding map} 
$\pi\colon S^{\Z}\to \cup\overline{J}$ by
\begin{equation}\label{def:pi}
\pi(\underline{a}):=x_0(\underline{a}).
\end{equation}
\noindent
Denote by $\sigma:S^\Z\to S^\Z$ the full left shift and let
$$
\check{S}:=\{\underline{a}\in S\colon x_n(\underline{a})\in J_{a_n}\text{ for all } n\in\Z\}.
$$
For any $\underline{a}\in S^\Z\setminus\check{S}$ there exists $n\in\Z$ such that $\pi\circ\sigma^n(\underline{a})\in \overline{J_{a_n}}\setminus J_{a_n}$. In particular, if all $J\in S$ are closed then $S^\Z\setminus\check{S}=\emptyset$.
However, this need not always be the case.

\begin{prop}\label{conjugacy}
The map $\pi$ given by \eqref{def:pi} has the following properties:
\begin{enumerate}
\item $\pi$ is well defined, continuous and for all $\underline{a}\in S^\Z$ one has
$$\pi\circ\sigma(\underline{a})=f^{\tau(J)}\circ\pi(\underline{a})
$$ where $J\in S$ is such that $\pi(\underline{a})\in \bar{J}$;
\item $\pi$ is one-to-one on $\check S$ and $\pi(\check S)=W$; 
\item if $\pi(\underline{a})=\pi(\underline{b})$  for some $\underline{a}, \underline{b}\in\check{S}$ then $a_n=b_n$ for all $n\ge 0$.
\end{enumerate}
\end{prop}
\begin{proof}
The proof follows directly from the definitions of $\pi$ and $\check S$ and Conditions (I1) and (I2).
\end{proof}
Consider the natural extension $(\tilde{W}, \tilde{F})$ of $(W,F)$ where 
$$
\tilde{W}:=\{\tilde{x}=(x_n)_{n\le 0}\colon F(x_n)=x_{n+1}\}$$
and $(\tilde{F}(\tilde{x}))_n:=x_{n+1}$. Condition (I2) implies that the coding map
$\tilde{\pi}: S^\Z\to\tilde{W}$ defined by
\begin{equation}\label{def:tildepi}
\tilde{\pi}(\underline{a}):=(x_n(\underline{a}))_{n\le 0}
\end{equation}
is a bijection. Proving the existence and uniqueness of equilibrium measures requires some additional condition on the inducing scheme 
$\{S,\tau\}$:
\begin{enumerate}
\bigskip
\item[{\bf (I3)}] The set $S^\Z\setminus\check S$ supports no (ergodic) $\sigma$-invariant measure which gives positive weight to any open subset.
\bigskip
\item[{\bf (I4)}] The induced map $F$ has at least one periodic point in $W$.
\end{enumerate}
\bigskip
Condition (I3) is designed to ensure that every Gibbs measure (defined in \eqref{eq:gibbs2}) is supported on $\check{S}$ and its projection by $\pi$ is thus supported on $W$ and is $F$-invariant. This projection is a natural candidate for the equilibrium measure for $F$. Condition (I4) is used to prove the finiteness of the pressure function (see Theorem~\ref{boundedenergy}).

\section{Thermodynamics of 2-sided countable shifts}\label{sec:therm-2-sided}

Our goal in this section is to carry out the thermodynamic formalism for countable shifts on the space $S^\Z$ of two-sided infinite sequences. 
Given $n < m \in \Z$ and $b_n, \cdots, b_m \in S$, a cylinder set is defined as
\[
	\mathcal{C}_n[b_n, \cdots, b_m] = \{ \underline{a} \in S^\Z: \quad a_i = b_i , n \le i \le m\}. 
\]
Consider a function
$\Phi:\, S^{\mathbb Z}\to{\mathbb R}$. Define the
\emph{n-variation} of $\Phi$ by
\[
V_n(\Phi) := \sup_{b_{-n+1}, \cdots, b_{n-1}\in S}\sup_{\underline{a},\underline{a}'\in \mathcal{C}_{-n+1}[b_{-n+1}, \cdots, b_{n-1}]}\{|\Phi(\underline{a})-\Phi(\underline{a}')|\}.
\]	
$\Phi$ is said to have \emph{summable variations} if
$$
\sum_{n\ge 1}V_n(\Phi)<\infty
$$
$\Phi$ has \emph{strongly summable variations} if
\begin{equation}\label{eq:sumnvar}
\sum_{n\ge 1}nV_n(\Phi)<\infty
\end{equation}
and $\Phi$ is \emph{locally H\"older continuous} if there exist $C>0$ and $0<r<1$ such that for all $n\geq 1$,
$$
V_n(\Phi)\leq C r^n.
$$
The \emph{Gurevich pressure} of $\Phi$ is defined by
$$
P_G(\Phi):=\lim_{n\rightarrow \infty}\frac{1}{n}\log\sum_{\sigma^n(\underline{a})=\underline{a}}\exp\left(\Phi_n(\underline{a})\right)\mathbb{I}_{[b]}(\underline{a})
$$
for some $b\in S$, where $\mathbb{I}_{[b]}$ denotes the characteristic function of the cylinder $[b]$ and
\[
\Phi_n(\underline{a}):=\sum_{k=0}^{n-1}\Phi(\sigma^k(\underline{a})).
\]
Since the Gurevich pressure only depends on the positive side of the sequences, $P_G(\Phi)$ exists whenever $\sum_{n\ge 1}V_n(\Phi)<\infty$ and it is independant of $b\in S$ by \cite[Theorem 1]{Sar99}.

Let $\mathcal{M}(\sigma)$ denote the set of $\sigma$ invariant Borel probabilities on $S^\Z$ and 
\begin{eqnarray*} {\mathcal M}_{\Phi}(\sigma) := \{\ \nu\in\mathcal{M}(\sigma) \colon\int_{S^{\mathbb Z}} \Phi d\nu > -\infty\ \}.
\end{eqnarray*}
A $\sigma$-invariant measure $\nu=\nu_{\Phi}$ is an \emph{equilibrium measure} for $\Phi$ provided
$$
 h_{\nu_{\Phi}}(\sigma) + \int\Phi d\nu_{\Phi} = \sup_{\nu\in{\mathcal M}_{\Phi}(\sigma)}\{\ h_{\nu}(\sigma) + \int \Phi d\nu\ \}.
$$
A measure $\nu=\nu_{\Phi}$ is a
\emph{Gibbs measure} for $\Phi$ provided there exists a constant $C_0>0$ such that for any (positive) cylinder set $C_0[b_{0}, \cdots, b_{n-1}]$ and any
$\underline{a}\in C_0[b_{0}, \cdots, b_{n-1}]$ we have
\begin{equation}\label{eq:gibbs2}
C_0^{-1}\le \frac{\nu(\mathcal{C}_0[b_{0}, \cdots, b_{n-1}])}{\exp(-nP_G(\Phi)+\Phi_n(\underline{a}))} \le C_0.
\end{equation}
We now present our result describing thermodynamics of the full shift of countable type on the space of two-sided sequences. 
\begin{thm}\label{gibbs2} Assume $\sup_{\underline{a}\in S^\Z}\Phi(\underline{a}) <\infty$ and 
$\Phi$ has strongly summable variations. Then
\begin{enumerate}
\item the variational principle for $\Phi$ holds, i.e.,
\[
P_G(\Phi)=\sup_{\nu\in{\mathcal M}_{\Phi}(\sigma)}\{\ h_{\nu}(\sigma) + \int \Phi d\nu\ \}.
\]
\item If $P_G(\Phi)<\infty$, then there exists a unique $\sigma$-invariant ergodic Gibbs measure $\nu_{\Phi}$ for $\Phi$.
\item If, furthermore, $h_{\nu_{\Phi}}(\sigma)<\infty$, then
$\nu_{\Phi}\in {\mathcal M}_{\Phi}(\sigma)$ and it is the unique equilibrium measure for $\Phi$.
\end{enumerate}
\end{thm}
The proof of the Theorem will require the two Lemmas stated below.

Consider a potential $\Phi^*: S^\inte\to\rl$ and define its $n$-variations $V_n(\Phi^*)$ and its Gurevich pressure $P_G(\Phi^*)$ as well as Gibbs and equilibrium measures as above but using one-sided sequences.
\begin{lem}\label{sarig}\cite{Sar03}
Assume $\sup_{\underline{a}\in S^\inte}\Phi^*(\underline{a}) <\infty$ and $\sum_{n\ge 2}V_n(\Phi^*)<\infty$. Then
\begin{enumerate}
\item the variational principle for $\Phi^*$ holds, i.e.,
\[
P_G(\Phi^*)=\sup_{\nu\in{\mathcal M}_{\Phi^*}(\sigma)}\{\ h_{\nu}(\sigma) + \int \Phi^* d\nu\ \}.
\]
\item If $P_G(\Phi^*)<\infty$ and $\sum_{n\ge 1}V_n(\Phi^*)<\infty$, then there exists a unique $\sigma$-invariant ergodic Gibbs measure
$\nu_{\Phi^*}$ for $\Phi^*$.
\item If, furthermore, $h_{\nu_{\Phi^*}}(\sigma)<\infty$, then
$\nu_{\Phi^*}\in {\mathcal M}_{\Phi^*}(\sigma)$ and is the unique equilibrium measure for $\Phi^*$.
\end{enumerate}
\end{lem}
Extending results for the full shift of countable type on the space of one-sided sequences to the full shift of countable type on the space of two--sided sequences can be achieved by a method inspired from \cite{Bow75} and \cite{Sin72}. 
\begin{lem}\label{bowentrick}
Assume $\Phi: S^\Z\to\R$ has strongly summable variations. Then there exists a bounded real valued function $u:S^\Z\to\R$ such that the function defined by
\[
\Psi(\underline{a}):=\Phi(\underline{a})+u(\sigma(\underline{a}))-u(\underline{a})
\]
satisfies: 
\begin{enumerate}
\item $\Psi(\underline{a})=\Psi(\underline{a}')$ whenever $a_i=a'_i$ for all $i\ge 0$;
\item $\Psi$ has summable variations, i.e., 
$\sum_{n\ge 1} V_n(\Psi)<\infty$;
\item if $\Phi$ is locally H\"older continuous then so is $\Psi$;
\item $P_G(\Psi)=P_G(\Phi)$.
\end{enumerate}
\end{lem}
\begin{proof}
Statements (1) and (2) can be proven under weaker assumptions (see \cite{CoeQua98} for finite shifts and \cite{Dao13} for countable shifts). 
We provide the proof nonetheless aiming at clarifying the proof of (3). 
Given a sequence $(r_k)_{k=-\infty}^0$ with  $r_k\in S$, define $r:S^\Z\to S^\Z$ by $(r(\underline{a}))_k=a_k$ for $k\ge 0$ and $(r(\underline{a}))_k=r_k$ for $k<0$ and set 
$$
u(\underline{a}):=\sum_{j=0}^\infty\  \Phi(\sigma^j(\underline{a}))-\Phi(\sigma^j(r(\underline{a}))).
$$
Since
$$
\left|\Phi(\sigma^j(\underline{a}))-\Phi(\sigma^j(r(\underline{a})))\right|\le V_{j+1}(\Phi),
$$
the function $u$ is well defined and uniformly bounded. For $\underline{a}, \underline{a}' \in \mathcal{C}_{-n+1}[b_{-n+1},\dots, b_{n-1}]$ one has
\begin{equation}\label{eq:u-summable}
\begin{aligned}
|u(\underline{a})-&u(\underline{a}')|\le \\
&\sum_{j=0}^{[\frac{n-1}2]}
|\Phi(\sigma^j(\underline{a}))-\Phi(\sigma^j(\underline{a}'))| 
+ |\Phi(\sigma^jr(\underline{a}))-\Phi(\sigma^jr(\underline{a}'))| \\
+&\sum_{j>[\frac{n-1}2]}^\infty
|\Phi(\sigma^j(\underline{a}))-\Phi(\sigma^jr(\underline{a}))|+|\Phi(\sigma^j(\underline{a}'))-\Phi(\sigma^jr(\underline{a}'))|\\
&\hspace{1cm}\le 4\sum_{j\ge [\frac{n-1}2]}V_j(\Phi).
\end{aligned}
\end{equation}
The strong summability of $\Phi$ yields
$$
\sum_{n\ge 1}V_n(u)<\infty.
$$
To prove the first statement, note that
$$
\Psi(\underline{a})=\sum_{j=0}^\infty \Phi(\sigma^j(r(\underline{a})))-\Phi(\sigma^jr(\sigma (\underline{a})))
$$
only depends on $a_i$ with $i\ge 0$. To prove the second statement, observe that for $\underline{a}, \underline{a}' \in \mathcal{C}_{-n+1}[b_{-n+1},\dots, b_{n-1}]$ the definition of $\Psi$ together with the arguments in \eqref{eq:u-summable} yield
$$
\left| \Psi(\underline{a})-\Psi(\underline{a}') \right| \le V_n(\Phi) + 8\sum_{j\ge [\frac{n-2}2]} V_j(\Phi)
$$
and thus
$$
\sum_{n\ge 1}V_n(\Psi)<\infty.
$$
This also proves the third statement of the Lemma, since if $V_n(\Phi)<C r^n$, then $V_n(\Psi)< C'r'^n$ for some $0<r'<1$ and $C'>0$.

The last statement follows from the definition of the Gurevich pressure.
\end{proof}

\begin{proof}[Proof of the Theorem]
Consider a function $\Phi$ satisfying the hypothesis of Theorem~\ref{gibbs2}, and let $\Psi^*: S^{\mathbb N}\rightarrow {\mathbb R}$ be defined by $\Psi^*(\underline{a}):=\Psi((\pi^*)^{-1}(\underline{a}))$, where $\pi^*:S^{\mathbb Z} \rightarrow S^{\mathbb N}$ is the canonical projection, and $\Psi=\Phi+u\circ\sigma-u$ is defined in Lemma~\ref{bowentrick}. The function $\Psi^*$ is well defined since $\Psi$ only depends on the positive elements of the sequence $(\pi^*)^{-1}(\underline{a})$.

In order to check that $\Psi^*$ satisfies the hypothesis of Lemma~\ref{sarig}, observe that
\begin{itemize}
\item[-] the boundedness of $u$ and $\sup_{\underline{a}\in S^\Z}\Psi(\underline{a})=\sup_{a\in S^\N}\Psi^*(\underline{a})$ imply that $\sup\Psi^*<\infty$ if and only if $\sup\Phi<\infty$. 
\item[-]
$\Psi^*$ has summable variations (over $S^\N$) since $\Psi$ has summable variations (over $S^\Z$) by Lemma~\ref{bowentrick}.
\item[-]
The Gurevich pressure $P_G(\Psi^*)$ (over $S^\N$) exists and $P_G(\Psi^*)=P_G(\Psi)=P_G(\Phi)$ since by definition the Gurevich pressure (over $S^\Z$) only depends on the positive side of the sequences.
\end{itemize}
Hence, if $\Phi$ satisfies the hypothesis of Theorem~\ref{gibbs2}, then 
$\Psi^*$ satisfies the hypothesis of Lemma~\ref{sarig}.

Since $\Phi$ and $\Psi$ are cohomologous, $\int\Phi \,d\nu=\int \Psi \,d\nu$ for any  $\nu\in\mathcal{M}(\sigma)$ and, in particular, 
${\mathcal M}_{\Phi}(\sigma)={\mathcal M}_{\Psi}(\sigma)$. Also, if 
$\nu^*:=(\pi^*)_*\nu$ then 
$\int_{S^{\mathbb Z}}\Psi \,d\nu=\int_{S^{\mathbb N}}\Psi^* \,d\nu^*$ and $h_\nu(\sigma)=h_{\nu^*}(\sigma)$. Moreover, for any given invariant Borel probability measure $\nu^*$ on $S^{\mathbb N}$ there exists a unique invariant Borel probability measure $\nu$ on $S^{\mathbb Z}$ with $\nu^*=(\pi^*)_*\nu$.

The variational principle then holds since $P_G(\Phi)=P_G(\Psi^*)$ and
$$ 
\sup_{\nu\in{\mathcal M}_{\Phi}(\sigma)}\{ h_{\nu}(\sigma)+\int \Phi\,d\nu\}
=\sup_{\nu^*\in{\mathcal M}_{\Psi^*}(\sigma)}\{ h_{\nu^*}(\sigma)+ 
\int\Psi^* \,d\nu^*\}.
$$
Let $\nu^*_{\Psi^*}$ be the unique Gibbs measure for $\Psi^*$ and denote its natural extension to the two-sided shift by $\nu_\Phi$. The one-to-one entropy preserving correspondence between invariant measures and their natural extensions implies that $\nu_\Phi$ is the unique Gibbs and equilibrium measure for $\Psi$ as well as for the cohomologous function $\Phi$.
\end{proof}

We now describe some ergodic properties of equilibrium measures. 
Consider a continuous transformation $T$. Recall that a $T$-invariant Borel probability measure $\mu$ has \emph{exponential decay of correlations} with respect to a class $\mathcal{H}$ of functions if there exists $0<\theta<1$ such that, for any $h_1, h_2\in\mathcal{H}$,
$$
\Big |\int h_1(T^n(x))h_2(x)\,d\mu -\int h_1(x)\,d\mu
\int h_2(x)\,d\mu\Big |\le K \theta^{n},
$$
for some $K=K(h_1,h_2)>0$.
The measure $\mu$ satisfies the {\it central limit theorem} (CLT)  with respect to the class $\mathcal{H}$ if 
there exists $\sigma\in\R$ such that 
$\displaystyle{\frac{1}{\sqrt{n}}\sum_{i=0}^{n-1}\left(h(T^ix)-\int
h\,d\mu\right)}$ converges in law to a normal distribution $\mathcal{N}(0,\sigma)$. Given a measure $\mu$, a function $h$ is \emph{cohomologous to a function $h'$} if there exists $g\in L^2(\mu)$ with $h-h'=g\circ f- g$, $\mu$-almost surely.

The following theorem describing some ergodic properties of the equilibrium measure $\nu_\Phi$ is a corollary of the well-known results by Ruelle \cite{Rue78} (see also \cite{Gor69}, \cite{Liv96} and \cite{Aar97}). 

\begin{thm}\label{sarig1}
Assume that the conditions of Theorem \ref{gibbs2} hold and that $\Phi$ is locally H\"older continuous. Then the unique equilibrium measure 
$\nu_\Phi$ has exponential decay of correlations, and satisfies the CLT
with respect to the class of bounded locally H\"older continuous functions with $\sigma>0$ iff $h$ is not cohomologous to a constant.
\end{thm}

\section{Thermodynamics of Inducing Schemes of Hyperbolic Type}\label{sec:therm-induced}

Assume the continuous map $f:X\to X$ of the compact metric space 
$(X, d)$ with $h_{top}(f)<\infty$ admits an inducing scheme $\{S,\tau\}$ 
of hyperbolic type. Set
\begin{equation}\label{setY}
Y:= \{f^k(x)\colon x\in W, \ 0\le k\le \tau(x)-1\}.
\end{equation}
Note that $Y$ is forward invariant under $f$. Denote the set of 
$f$-invariant ergodic Borel probability measures on $Y$ by 
$\mathcal{M}(f,Y)$ and the set of $F$-invariant ergodic Borel probability measures on $W$ by $\mathcal{M}(F,W)$. For any 
$\nu\in\mathcal{M}(F, W)$ let
$$
Q_\nu:=\int_W\tau\, d\nu.
$$
If $Q_\nu<\infty$ the \emph{lifted measure} ${\mathcal L}(\nu)$ is defined as follows: for any $E\subset Y$,
$$
{\mathcal L}(\nu)(E):= \frac{1}{Q_\nu}\sum_{J\in S}\sum_{k=0}^{\tau(J)-1}\nu(f^{-k}(E)\cap J).
$$
\begin{prop}
If $\nu\in\mathcal{M}(F,W)$ and $Q_\nu<\infty$ then
${\mathcal L}(\nu)\in\mathcal{M}(f,Y)$, ${\mathcal L}(\nu)(W)=\frac{1}{Q_\nu}$.
\end{prop}
Consider the class of measures
\begin{equation}
\begin{aligned}\label{L-1}
\mathcal{M}_L(f,Y):=\{\mu\in\mathcal{M}(f,Y): \text{there is }&\nu:=\mathcal{L}^{-1}(\mu)\in\mathcal{M}(F,W)\\&
\text{ with }{\mathcal L}(\nu)=\mu\}.
\end{aligned}
\end{equation}
A measure $\mu\in\mathcal{M}_L(f,Y)$ is called \emph{liftable} and $\mathcal{L}^{-1}(\mu)$ is an \emph{induced measure} for $\mu$.
\begin{rmk}
The class of liftable measures depends on the choice of the inducing scheme $\{S,\tau\}$: a given measure in ${\mathcal M}(f,Y)$ may be liftable with respect to one inducing scheme but not with respect to another. One can show (see \cite{Zwe05}) that if $\tau\in L^1(Y, \mu)$, then $\mu\in\mathcal{M}_L(f,Y)$.
\end{rmk}
A measure $\mu_\varphi\in{\mathcal{M}_L(f,Y)}$ will be called an \emph{equilibrium measure for} $\varphi$ (in the space 
$\mathcal{M}_L(f,Y)$ of \emph{liftable} measures) if
$$
P_L(\varphi):=\sup_{\mathcal{M}_L(f,Y)}\{h_\mu(f)+\int_X \varphi \,d\mu\} = h_{\mu_\varphi}(f)+\int_X\varphi \,d\mu_\varphi.
$$
This definition of equilibrium measures differs from the standard one as it only considers liftable measures and hence on the inducing scheme.

Condition (I4) on the inducing scheme of hyperbolic type implies:
\begin{thm}[\cite{PesSen08}, Theorem 4.2]\label{boundedenergy}
Assume that the inducing scheme $\{S, \tau\}$ satisfies Condition (I4), 
the function $\Phi$ has strongly summable variations and 
$P_G(\Phi)<\infty$. Then $-\infty<P_L(\varphi)<\infty$.
\end{thm}

The existence of a unique equilibrium measure for $\varphi$ is proven by first studying the problem for the induced system $(F,W)$ and the \emph{induced potential}
$\overline{\varphi}: W\to\rl$ defined by
$$
\overline\varphi(x):= \sum_{k=0}^{\tau(x)-1}\varphi(f^k(x)).
$$
The crucial relation between the original and induced systems is given by Abramov's and Kac's formul{\ae} which extend verbatim to inducing schemes of expanding type
\begin{prop}\cite[Theorem 2.3]{PesSen08}
\label{abramovkac}
Let $\nu\in\mathcal{M}(F,W)$ and $Q_\nu<\infty$. Then
$$
h_\nu(F)=Q_\nu\cdot h_{{\mathcal L}(\nu)}(f).
$$
Let $\varphi$ be a potential and $\overline\varphi$ the corresponding induced potential. If
$\int_W\overline\varphi \,d\nu<\infty$ then
$$
-\infty < \int_W\overline\varphi\,d\nu =Q_\nu\cdot\int_X\varphi \,d{\mathcal L}(\nu) <\infty.
$$
\end{prop}
The study of the existence and uniqueness of equilibrium measures for the induced system $(F,W)$ is carried out by conjugating the induced system to the two-sided full shift over the countable alphabet $S$. This requires that the potential function $\Phi:=\bar\varphi\circ \pi$ be well defined on $S^\Z$, where $\pi$ is the coding map defined by \eqref{def:pi} (see also Proposition \ref{conjugacy}). To this end we require that
\begin{enumerate}
\bigskip
\item[{\bf (P1)}]
the induced potential $\overline{\varphi}$ can be extended by continuity to a function on $\bar{J}$ for every $J\in S$.
\bigskip
\end{enumerate}
Denote the potential induced by the \emph{normalized} potential
$\varphi-P_L(\varphi)$ by
$$
\varphi^+:=\overline{\varphi-P_L(\varphi)}=\overline{\varphi}-P_L(\varphi)\tau
$$
and let $\Phi^+:=\varphi^+\circ\pi$ where $\pi$ is the coding map from \eqref{def:pi}.
\begin{thm}\label{liftgibbs}
Let $\{S,\tau\}$ be an inducing scheme of hyperbolic type satisfying Conditions (I3) and (I4) and $\varphi$ a potential satisfying Condition (P1). Assume that $\Phi$ has strongly summable variations and $P_G(\Phi)<\infty$. Assume also that $P_G(\Phi^+)<\infty$ and $\sup_{a\in S^\Z}\,\Phi^+(a)<\infty$. Then
\begin{enumerate}
\item [(a)] there exists a $\sigma$-invariant ergodic Gibbs measure 
$\nu_{\Phi^+}$ for $\Phi^+$;
\item[(b)] if $h_{\nu_{\Phi^+}}(\sigma)<\infty$ then 
$\nu_{\Phi^+}$ is the unique equilibrium measure for 
$\Phi^+$;
\item[(c)] if $h_{\nu_{\Phi^+}}(\sigma)<\infty$, then the measure 
$\nu_{\varphi^+}:=\pi_*\nu_{\Phi^+}$ is a unique $F$-invariant ergodic equilibrium measure for $\varphi^+$;
\item[(d)] if $P_G(\Phi^+)=0$ and $Q_{\nu_{\varphi^+}}<\infty$ then $\mu_\varphi={\mathcal L}(\nu_{\varphi^+})$ is the unique equilibrium ergodic measure in the class $\mathcal{M}_L(f, Y)$ of liftable measures.
\end{enumerate}
\end{thm}
\begin{proof}
Statements (a) and (b) follow from Theorem~\ref{gibbs2} after noting that by  Theorem~\ref{boundedenergy}, $P_L(\varphi)$ is finite and that $V_n(\Phi)=V_n(\Phi^+)$.

To prove Statement (c), consider the function $\phi:\tilde{W}\to\R$ defined by $\phi((x_n)_{n\le 0}):=\overline{\varphi}(x_0)$ and 
$\Phi=\phi\,\circ\,\tilde{\pi}$ and $\Phi^+=\phi^+\circ\ \tilde{\pi}$ where 
$\phi^+=\phi-P_L(\varphi)\tau$ and $\tilde{\pi}:S^\Z\to\tilde{W}$ is the projection defined by~\eqref{def:tildepi}. By Condition (I3), the measure $\nu_{\Phi^+}$ gives full weight to the set $\check{S}$. Proposition~\ref{conjugacy} then implies that the measure 
$\tilde{\nu}_{\phi^+}:=\tilde{\pi}_*\nu_{\Phi^+}$ is the unique 
$\tilde{F}$-invariant ergodic equilibrium measure for $\phi^+$, i.e.,
$$
h_{\tilde{\nu}_{\phi^+}}(\tilde{F})+\int\phi^+\,d \tilde{\nu}_{\phi^+}=\sup_{\nu\in\mathcal{M}( \tilde{F}, \tilde{W})}\{h_\nu(\tilde{F})+\int\phi^+\,d\nu\},
$$
where $\mathcal{M}(\tilde{F}, \tilde{W})$ denotes the set of ergodic 
$\tilde{F}$-invariant Borel probabilities on $\tilde{W}$. 

Recall now that there is a one-to-one entropy preserving correspondence between $\mathcal{M}(F, W)$ and 
$\mathcal{M}(\tilde{F}, \tilde{W})$.
Also, for any 
$\overline{\varphi}\in L^1(\nu)$ the function $\phi$ on $\tilde{W}$ 
satisfies $\int\overline{\varphi}\,d\nu=\int\phi\,d\tilde{\nu}$. Statement (c) follows since $\nu_{\varphi^+}$ is the image of $\tilde{\nu}_{\phi^+}$ under this correspondence.

To prove Statement (d) observe that $Q_{\nu_{\varphi^+}}<\infty$ and Proposition~\ref{abramovkac} imply $h_{\nu_{\Phi^+}}(\sigma)<\infty$. Statement (c) and Proposition~\ref{abramovkac} imply
$$0=h_{\varphi^+}(F)+\int\varphi^+\,d\nu_{\varphi^+}=\frac{1}{Q_{\nu_{\varphi^+}}}\left(h_{\mu_\varphi}(f)+\int \varphi\,d\mu_\varphi-P_L(\varphi)\right)$$ 
hence
$$
h_{\mu_\varphi}(f)+\int \varphi\,d\mu_\varphi=P_L(\varphi).
$$
\end{proof}

To describe the ergodic properties of equilibrium measures assume that $\nu_{\varphi^+}$ has \emph{exponential tail}: there exist $C>0$ and $0<\theta<1$ such that for all $n>0$,
$$\nu_{\varphi^+}(\{x\in W: \tau(x)\ge n\})\le C\theta^n.$$

\begin{thm}\label{ergodic}
Assume the conditions of Theorem \ref{liftgibbs} hold. Assume also that the induced function $\overline\varphi$ on $W$ is locally H\"older continuous. If $\nu_{\varphi^+}$ has exponential tail then the unique equilibrium measure $\mu_\varphi$ has exponential decay of correlations and satisfies the CLT with respect to the class of functions whose induced functions on $W$ are bounded locally H\"older continuous.
\end{thm}
\begin{proof} See \cite[ Theorem 4.6]{PesSen08}.
\end{proof}

We describe some \emph{verifiable} conditions on the potential function 
$\varphi$ under which the assumptions of Theorem \ref{liftgibbs} hold:
\begin{enumerate}
\bigskip
\item[{\bf (P2)}]
there exist $C>0$ and $0<r<1$ such that for any $n\ge 1$
$$
V_n(\phi):=V_n(\Phi)\le C r^n;
$$
\bigskip
\item[{\bf (P3)}]
\[
\sum_{J\in S}\,\sup_{x\in J}\exp \bar\varphi(x)<\infty;
\]
\bigskip
\item[{\bf (P4)}]
there exists $\epsilon>0$ such that
\[
\sum_{J\in S}\tau(J)\sup_{x\in J}\exp(\varphi^+(x)+\epsilon\tau(x))<\infty.
\]
\bigskip
\end{enumerate}
\begin{thm}\label{conditions} Let $\{S, \tau\}$ be an inducing scheme of hyperbolic type satisfying Conditions (I3) and (I4). The following statements hold:
\begin{enumerate}
\item Condition (P2) implies that the function $\Phi$ has strongly summable variations;
\item Condition (P3) implies the finiteness of the Gurevich pressure $P_G(\Phi)<\infty$ and $\sup_{a\in S^\Z}\,\Phi(a)<\infty$;
\item Condition (P4) implies that $\sup_{a\in S^\Z}\,\Phi^+(a)<\infty$ and $P_G(\Phi^+)=0$;
\item If $\Phi$ has strongly summable variations, then Condition (P3) implies $Q_{\nu_{\varphi^+}}<\infty$.
\end{enumerate}
\end{thm}
\begin{proof} See the proof in \cite[Theorem 4.5]{PesSen08}.
\end{proof}
The existence and uniqueness of equilibrium measures is now an immediate corollary of Theorems \ref{liftgibbs}, \ref{ergodic} and \ref{conditions}.
\begin{thm}\label{equilibrium1}
Let $\{S,\tau\}$ be an inducing scheme of hyperbolic type satisfying Conditions (I3) and (I4). Assume that the potential function $\varphi$ satisfies Conditions (P1)-(P4). Then 
\begin{enumerate}
\item there exists a unique equilibrium measure 
$\mu_\varphi$ for $\varphi$ among all measures in 
$\mathcal{M}_L(f, Y)$; 
the measure $\mu_\varphi$ is ergodic;
\item if $\nu_{\varphi^+}=\mathcal{L}^{-1}(\mu_\varphi)$ has exponential tail, then $\mu_\varphi$ has exponential decay of correlations and satisfies the CLT with respect to a class of functions which contains all  bounded locally H\"older continuous functions on $Y$.
\end{enumerate} 
\end{thm}

\section{Liftability of invariant measures}
\label{sec:liftability}

In this section we describe some conditions under which all ergodic invariant measures of sufficiently large entropy are liftable.  

Denote the set of elements of the inducing scheme $\{S, \tau\}$ for which the inducing time is the \emph{first return time} by
$$
S'=\{J\in S: f^i(J)\cap W=\emptyset \text{ for }0<i<\tau(J)\}$$
and the set of those whose return time is \emph{not} a first return time by $S^*:=S\setminus S'$. Let 
$$S(n):= \{J\in S:\tau(J)=n\}$$
and $S'(n):=S'\cap S(n)$ and $S^*(n):=S^*\cap S(n)$. Also let $S_n:=\sharp\, S(n)$ denote the cardinality of $S(n)$ and define $S^*_n$ and $S'_n$ similarly. 

\begin{thm}\label{liftability}
Assume the following conditions hold:
\begin{enumerate}
\item[(L1)] For any $\epsilon>0$, there exist numbers $l=l(\epsilon)$ and $N=N(\epsilon)$ such that for any
$\tau(J)>N$,
$$
\text{diam}(f^k(J))<\epsilon \, \text{ for } k=\frac{l}2, \dots, \tau(J)-\frac{l}2-1;
$$
\item[(L2)] There exists $h>0$ such that for all sufficiently large $n$,
\begin{equation}\label{eq:element-count}
S^*_n\le e^{hn}.
\end{equation}
\end{enumerate}
Then any ergodic $\mu\in\mathcal{M}(f, X)$ with $\mu(W)>0$ and $h_\mu(f)>h$ is liftable.
\end{thm}

The proof follows ideas from Keller \cite{Kel89} and Zweim\"{u}ller \cite{Zwe05}.
\begin{proof}
Define
$\hat{Y}:=\{(x,k)\colon x\in W,\ 0\le k\le \tau(x)-1\} \subset W\times\N$ and $\hat{f}:\hat{Y}\to \hat{Y}$ by
$$
\hat{f}(x,k):=\begin{cases}
(x,k+1), & k<\tau(x)-1 \\ (F(x),0), & k=\tau(x)-1.
\end{cases}
$$
Define
$\hat{\mu}_0:=i_*\mu|_W$ where $i:W\to \hat{Y}$ is given by $i(x)=(x,0)$ and
$$
\hat{\mu}_n:=\frac{1}{n}\sum_{j=0}^{n-1}\hat{f}^j_*\hat{\mu}_0.
$$
Let $\bar{\mu}_0=\pi_1^*\mu$, where $\pi_1(x,k)=x$ is the projection to the first component. One can show as in the proof of Theorem 1.1 of \cite{Zwe05} that $\hat{\mu}_{n}\ll \bar{\mu}_0$ and $\frac{d\hat{\mu}_{n}}{d\bar{\mu}_0}\le 1$ implying the existence of a function $\hat{\chi}\in L^1(\bar{\mu}_0)$ and a subsequence $n_k$ such that $\frac{d\hat{\mu}_{n_k}}{d\bar{\mu}_0}|_{\hat{E}}\stackrel{weakly}\to \hat{\chi}$ for any bounded set $\hat{E}$, where $\hat{E}$ is bounded if there exists $M>0$ such that $\sup\{k\colon (x,k)\in\hat{E}\}\le M$. As a consequence, for $\hat{\mu}:=\hat{\chi}\cdot \bar{\mu}_0$, one has that $\hat{\mu}_{n_k}(\hat{E})\to \hat{\mu}(\hat{E})$ for any bounded $\hat{E}$. Moreover, if $\hat{\chi}\ne 0$, then the measure $\nu$ defined by
$$
\nu(E):=\sum_{J\in S}\hat{\mu}((J\cap E)\times\{0\})
$$
satisfies ${\mathcal L}(\nu)=\mu$ (see \cite{Zwe05}, (3.2)).

Therefore the statement follows from the
following:
\begin{prop}\label{smallentropy}
Given $\mu\in\mathcal{M}(f, X)$, suppose there exists a sequence $n_k$ such that $\frac{d\hat{\mu}_{n_k}}{d\bar{\mu}_0}|_{\hat{E}}\stackrel{weakly}\to 0$ for any bounded set $\hat{E}$, then $h_{\mu}(f)\le h$.
\end{prop}
The proof of the Proposition follows from Lemmas~\ref{entropybound1} and \ref{entropybound2}. But in order to prove these Lemmas, one needs the following:

\begin{lem}\label{escaping-mu-mass}
Let
$$
\hat{E}_N:=\Big\{(x,k)\colon x\in W\ \mbox{ and }\ 0\le k \le N\Big\}$$
and
$$\hat{F}':=\Big\{(x, k)\colon x\in\bigcup_{J\in S'}J \ \mbox{ and }\ 0\le k\le \tau(J)-1\Big\}. 
$$
Suppose $\frac{d\hat{\mu}_{n_k}}{d\bar{\mu}_0}|_{\hat{E}}\stackrel{weakly}\to 0$  over all bounded set $\hat{E}$, then 
$\displaystyle{\lim_{k\to \infty}\hat{\mu}_{n_k}(\hat{E}_N\cup\hat{F}')=0}$ for all $N$.
\end{lem}
\begin{proof}
The sets $\hat{E}_N$ are bounded, therefore
$\displaystyle{\lim_{k\to \infty}\hat{\mu}_{n_k}(\hat{E}_N)= 0}$.

Since for any $J\in S'(n)$ the first return time is $n$, Kac's formula yields 
$$   
\sum_{n\ge 1}\sum_{J\in S'(n)}n\mu(J)
\le 1<\infty
$$
and thus
$$
\lim_{N\to \infty}\sum_{n> N}\sum_{J\in S'(n)}n\mu(J)=0. 
$$
Also, the fact that $\frac{d\hat{\mu}_{n_k}}{d\bar{\mu}_0}\le 1$ implies
$$
\begin{aligned} 
\hat{\mu}_{n_k}(\hat{F}'\setminus \hat{E}_N)&=\hat{\mu}_{n_k}\Bigl(\bigcup_{n>N}\bigcup_{J\in S'(n)}\bigcup_{k=N}^{n-1}(J\times\{k\})\Bigr) \\
&\le\bar{\mu}_0\Bigl(\bigcup_{n>N}\bigcup_{J\in S'(n)}\bigcup_{k=N}^{n-1}(J \times\{k\})\Bigr)\le\sum_{n> N}\sum_{J\in S'(n)}n \mu(J). 
\end{aligned}
$$
The proof now follows by contradiction: assume $\displaystyle{\lim_{k\to\infty}\hat{\mu}_{n_k}(\hat{F}')=c >0}$. Choose $N\in\N$ so large that 
$$
\sum_{n> N}\sum_{J\in S'(n)}n \mu(J)\le \frac{c}2.
$$
Therefore $\hat{\mu}_{n_k}(\hat{F}'\cap \hat{E}_N)\ge \frac{c}2$ for all $n_k$, contradicting the fact that $\displaystyle{\lim_{k\to \infty}\hat{\mu}_{n_k}(\hat{E}_N)= 0}$ for any $N$. 
\end{proof}

\begin{lem}\label{bigZ}
Consider $A_{n}^N(x): W\to \R$ defined by
$$
A_n^N(x):=\frac1n\sharp\,\{0\le j\le n-1: \,\hat{f}^j(i(x))\in\hat{E}_N\cup \hat{F}'\}.
$$
If $\displaystyle{\lim_{k\to\infty}\frac{d\hat{\mu}_{n_k}}{d\bar{\mu}_0}=0}$ then for every $\varepsilon>0$ there exists $Z=Z(n_k, \varepsilon)\subset W$ \emph{independent of $N$} such that $\mu(Z)\ge (1-\varepsilon)\mu(W)$ and
$\lim_{k\to\infty}A_{n_k}^N(x)=0$ uniformly on $Z$ \emph{for all $N$}.
\end{lem}
\begin{proof}
By formula (10) of \cite{PesZha07} $\lim_{k\to\infty}\hat{\mu}_{n_k}(\hat{E}_N\cup\hat{F}')=0$ implies $\lim_{k\to\infty}\int_WA_{n_k}^N(x) \,d\mu(x)= 0$. For $N=1$, by passing to a subsequence if necessary, there exists a set $Z_1$ with $\mu(Z_1)\ge(1-\frac{\varepsilon}{2})\mu(W)$ and $A^1_{n_k}(x)\to 0$ uniformly on $Z_1$. Since
$\lim_{k\to\infty}\int_{Z_1}A^2_{n_k}(x)\,d\mu(x)=0$, again by passing to a proper subsequence, there exists $Z_2\subset Z_1$ with $\mu(Z_2)\ge (1-\frac{\varepsilon}{2}-\frac{\varepsilon}{4})\mu(W)$ and
$\lim_{k\to\infty}A^2_{n_k}(x)=0$ uniformly on $E_2$. By a diagonal argument, there exists a subsequence, still denoted by $n_k$, and a set 
$Z_N\subset W$ with 
$$
\mu(Z_N)\ge (1-\sum_{j=0}^N\frac{\varepsilon}{2^j})\mu(W)
$$
such that $\lim_{k\to\infty}A^N_{n_k}(x)=0$ uniformly on $Z_N$.
Let $Z=\bigcap_{k\ge 1}Z_k$. We have that
$\mu(Z)\ge(1-\varepsilon)\mu(W)$ and $A^N_{n_k}(x)\to 0$ uniformly on $Z$ for every fixed $N$.
\end{proof}
The claim now follows from
\begin{equation}\label{eq:ent-est}
h_\mu(f) \le h_{\text{top}}(f,Z) \le h
\end{equation}
where $h_{\text{top}}(f,Z)$ is the topological entropy of the map $f|Z$ restricted to the non-compact set $Z$. For the reader's convenience, 
we briefly recall the definition of this notion of topological entropy from \cite{Bow73} as presented in \cite[Section 11]{Pes97} before proving the inequalities in Lemmas~\ref{entropybound1} and \ref{entropybound2}.

The set $B_n(x,\epsilon):=\{y\in Y: d(f^ix,f^iy)<\epsilon, \, 0\le i\le n-1\}$ for 
$x\in Y$, $n>0$ and $\epsilon>0$ is called a \emph{Bowen ball}. Given a set $E\subset Y$ and numbers $\lambda \in \R,\ \epsilon>0$ and $K\in\N$, let
$$
M(E,\lambda,\epsilon,K):=\inf_{{\mathcal G}}\Big\{\sum_{B_{n}(x,\epsilon)\in {\mathcal G}}e^{-\lambda n}\Big\},
$$
where the infimum is taken over all covers ${\mathcal G}$ of $E$ whose elements are Bowen balls $B_{n}(x,\epsilon)$ with $n\ge K$. We stress that different balls $B_n(x,\epsilon)$ in the cover $\mathcal{G}$ may have different length $n$.
Then write
$$
m(E,\lambda,\epsilon):=\lim_{K\to \infty}M(E,\lambda, \epsilon,K)
$$
and let
$$
h_{top}(f, E, \epsilon):=\inf_{\lambda\in\R}\ \big\{ m(E,\lambda,\epsilon)=0\big\}
$$
Finally define
$$
h_{top}(f, E):= \lim_{\epsilon\to 0}\ h_{top}(f, E, \epsilon).
$$
\begin{lem}\label{entropybound1}
Let $\mu\in\mathcal{M}(f, X)$ with $\mu(E)>0$. Then 
$h_\mu(f)\le h_{\text{top}}(f,E)$.
\end{lem}
\begin{proof}
If $\mu(E)=1$, the statement follows from the inverse variational principle (see \cite{Pes97}):
$$
h_\mu(f)=\inf\{h_{top}(f,E):\mu(E)=1\}.
$$
If $0<\mu(E)<1$, the set  $\bigcup_{i=0}^\infty f^i(E)$ is invariant and thus has full measure and
$$
h_{top}\left(f,\ \bigcup_{i=0}^\infty f^i(E)\right)
=\sup_{i\ge 0}h_{top}\left(f, f^i(E)\right)=h_{top}(f, E).
$$
The last equality is due to the fact that
$h_{top}\left(f, f(E)\right)\le h_{top}(f, E)$.
\end{proof}
The first inequality in \eqref{eq:ent-est} now follows from Lemmas~\ref{bigZ} and \ref{entropybound1}. It remains to prove the second inequality in \eqref{eq:ent-est}
\begin{lem}\label{entropybound2}
$h_{top}(f,Z)\le h$
\end{lem}
The strategy is to obtain an appropriate upper bound estimate of $M(Z,\lambda,\epsilon,n_k)$ for a particular covering.

Given $\epsilon>0$, choose a sufficiently large 
$N>N(\epsilon)$ such that \eqref{eq:element-count} holds for all $n>N$. To simplify notation write $\tau_i=\tau(J_i)$. 

A \emph{designation} is defined as a pair of $(s+1)$-tuples 
$$
(\xi, \eta)=(\xi = (m_0,\dots, m_s), \eta = (J_0, \dots, J_s))
$$ 
satisfying the following conditions: 
\begin{itemize}
\item $0\le m_0 < m_1< \dots < m_s \in \N$. 
\item For each $0\le j \le s$, $J_j\in S^*$ and $\tau_j >N$. 
\item For each $0\le j \le s-1$, we have $m_j+\tau_j\le m_{j+1}$. 
\end{itemize}
A designation $(\xi, \eta)$ defines a subset $V(\xi, \eta)$ of $X$ by designating the $m_j$-th iterate of $x$ to be contained in $J_j$, i.e.,
$$ 
V(\xi,\eta)=\{x\in X: f^{m_j}(x)\in J_j \ \mbox{ for all } \ 0\le j\le s\}. 
$$
This can be viewed as a \emph{partial coding} of the dynamics using \emph{only} sets from $S^*$ with inducing time larger than $N$. 

We say that $0\le i\le m_s+\tau_s$ is \emph{designated} if there exists $0\le j< s$ satisfying $m_j\le i< m_j+\tau_j$. Otherwise, it is called \emph{undesignated}. The total number of undesignated indices from $0$ to $m_s$ is 
$$
m_0 + \sum_{j=0}^{s-1}\left( m_{j+1} - m_j - \tau_j\right) = m_s - \sum_{j=0}^{s-1}\tau_j. 
$$
For a small $\delta>0$ let
${\mathcal C}:={\mathcal C}(\epsilon, N,\delta, n_k)$ denote the set of all designations $(\xi, \eta)$ satisfying 
\begin{equation}\label{eq:J-i}
m_{s-1} + \tau_{s-1} <n_k\le m_s + \tau_s.
\end{equation}
and the total undesignated space is small:
\begin{equation}\label{eq:J-i1}
m_s - \sum_{j=0}^{s-1}\tau_j<\delta n_k.
\end{equation}

\begin{lem}\label{Z-covering}
For $k\in\N$ sufficiently large one has
$$
Z\subset \bigcup_{(\xi, \eta)\in\mathcal{C}}V(\xi, \eta).
$$
\end{lem}
\begin{proof}
It follows from the definition of ${\mathcal C}$ that for any $x\in W$ with $A_{n_k}^N(x)\le\delta$ one has 
$x\in V(\xi, \eta)$ for some $(\xi, \eta)\in{\mathcal C}$. By Lemma~\ref{bigZ} $A_{n_k}^N(x)\to 0$ uniformly on $Z$ yielding the statement.
\end{proof}
\begin{proof}[Proof of Lemma~\ref{entropybound2}]
Fix $\epsilon, N, \delta$ and the sequence $n_k$. 
Write ${\mathcal C}={\mathcal C}_1\cup{\mathcal C}_2$ where
\begin{itemize}
\item $(\xi,\eta)\in{\mathcal C}_1$ if $n_k$ is undesignated, i.e., 
$m_{s-1}+\tau_{s-1}\le n_k<m_s$;
\item and $(\xi,\eta)\in{\mathcal C}_2$ if $n_k$ is designated, i.e., 
$m_s \le n_k < m_s + \tau_s$. 
\end{itemize}
We shall cover the set $Z$ with Bowen balls by covering each set 
$V(\xi,\eta)$ with $(\xi,\eta)\in\mathcal{C}_1$ and 
$(\xi,\eta)\in\mathcal{C}_2$. We begin by covering $V(\xi,\eta)$ with 
$(\xi,\eta)\in\mathcal{C}_1$. Let 
$$
\beta(\epsilon):=\min_{\mathcal{G}}\ \sharp\,\mathcal{G},
$$ 
where the minimum is taken over all coverings $\mathcal{G}$ of $X$ by 
$\epsilon$-balls $B$ (in the original metric $d$). Since the space $X$ is compact $\beta(\epsilon)$ is finite. For  $x\in V(\xi, \eta)$ for any 
$(\xi, \eta)\in{\mathcal C}$, each undesignated iterate can be covered by $\beta(\epsilon)$ number of $\epsilon$-balls. There are at most $\delta n_k$ such iterates from the definition of $\mathcal{C}$. For the remaining iterates, the hypothesis (L1) implies that for $J\in S$ with $\tau(J)>N>N(\epsilon)$ and for $j=\frac{l(\epsilon)}2, \dots,\tau(J)-\frac{l(\epsilon)}2 -1$ the sets $f^j(J)$ are contained in a single ball of radius $\epsilon$.
Given any $(\xi, \eta)\in{\mathcal C}_1$, the minimum number of $(n_k,\epsilon)$-balls needed to cover $V=V(\xi, \eta)$ is bounded by the product of the minimum  number of $\epsilon$-balls needed to cover each iterate, i.e.,
\[
	\prod_{j=0}^{n_k} \min\{\ \ell\colon\,  \bigcup_{i=1}^{\ell} B(x_i, \epsilon) \supset f^j(V)\,\}
\]
which in turn is bounded by 
\begin{equation}\label{eq:C2-mult}
\beta(\epsilon)^{\delta n_k+(1-\delta)\frac{n_k}{N}l(\epsilon)}<\beta(\epsilon)^{n_k(\delta+ \frac{l(\epsilon)}{N})}.
\end{equation}

For the collection $\mathcal{C}_2$ the situation is different: although each set $V(\xi, \eta)$ with 
$(\xi, \eta)\in{\mathcal C}_2$ can be covered by a finite number of $(n_k,\epsilon)$-balls, the total number of elements in 
$\mathcal{C}_2$ is infinite (there is no upper bound for $m_s + \tau_s$). One therefore needs to proceed differently. For any $(\xi, \eta)\in{\mathcal C}_2$ the set
$V(\xi, \eta)$ is covered by 
$(m_s + \tau_s,\epsilon)$-balls where the \emph{"length"} of each ball will depend on $\tau_s$.

By the same counting argument as in the $\mathcal{C}_1$ case presented above, for any $(\xi, \eta)\in {\mathcal C}_2$ each $V(\xi, \eta)$ can be covered by at most
\begin{equation}\label{eq:C1-mult}
\beta(\epsilon)^{n_k( \delta + \frac{l(\epsilon)}{N})}
\end{equation}
$(m_s+\tau_s,\epsilon)$-balls.

Let ${\mathcal G}={\mathcal G}_1\cup {\mathcal G}_2$, where
${\mathcal G}_1$ is the collection of $(n_k,\epsilon)$-balls covering all sets $V(\xi, \eta)$ with $(\xi, \eta)\in{\mathcal C}_1$ and ${\mathcal G}_2$ is the collection of 
$(m_s+\tau_s,\epsilon)$-balls covering all sets 
$V(\xi, \eta)$ with $(\xi, \eta)\in{\mathcal C}_2$. Then
$$
M(Z,\lambda,\epsilon,n_k)\le\sum_{B_{n_k}(x,\epsilon)\in{\mathcal G}}e^{-\lambda n_k}
=\sum_{B_{n_k}(x,\epsilon)\in{\mathcal G}_1}e^{-\lambda n_k}
+\sum_{B_{n_k}(x,\epsilon)\in{\mathcal G}_2}e^{-\lambda n_k}.
$$
The two sums in the above formula are estimated separately. All Bowen balls of the collection $\mathcal{G}_1$ have the same length $n_k$. 
Equation (\ref{eq:C2-mult}) yields
$$
\sum_{B_{n_k}(x,\epsilon)\in {\mathcal G}_1}e^{-\lambda n_k}
\le\sharp\,(\mathcal{C}_1)\ \beta(\epsilon)^{n_k (\delta  +\frac{ l(\epsilon)}{N})}\ e^{-n_k \lambda}.
$$
One therefore only needs to estimate the cardinality of 
${\mathcal C}_1$. Consider an element
$(\xi, \eta)\in\mathcal{C}_1$ with $l$ undesignated spaces, i.e.,
$$
l= m_s - \sum_{j=0}^{s-1} \tau_j  < \delta n_k
$$
where the inequality follows from Equation~\eqref{eq:J-i1}.
The elements of $\mathcal{C}_1$ are counted as follows. First partition $(0, \dots, n_k-l)$ into intervals of minimal length 
$N$, corresponding to the partition of $\{0, \cdots, \tau_0 + \cdots + \tau_{s-1}\}$ into intervals of designated iterates of length $\tau_0, \cdots, \tau_{s-1}$.  By Condition (L2) there are at most $e^{h\tau_i}$ choices of elements with $J_i\in S^*$ and $\tau_i>N$. Then insert a total of at most $l$ intervals corresponding to the undesignated space. With the convention that undefined binomial formulas are zero, one then has
$$
\begin{aligned}
\sharp\,(\mathcal{C}_1)
&\le\sum_{i=1}^{n_k}\sum_{l=1}^{\delta n_k} \sum_{s=1}^{\frac{n_k-l}{N}} \binom{l}{i-s}\binom{n_k-l}{s} e^{h(n_k-l)}\\
&\le\sum_{i=1}^{n_k}\sum_{l=1}^{\delta n_k} \sum_{s=1}^{n_k/N}\binom{l}{i-s}\binom{n_k}{s}  e^{h(n_k-l)}.
\end{aligned}
$$
By Stirling's formula (see e.g. \cite{Fel68}),
\begin{equation}\label{stirling}
\binom{n}{m}\le Ce^{\sigma(m/n)\cdot n},
\end{equation}
where $\sigma(x)=-x\log x -(1-x)\log(1-x)$. Since $\lim_{x\to 0}\sigma(x)=0$ and $\sigma'(x)=\log(\frac1x-1)$, we have that 
$0\le\sigma(x)\le \log 2$ for any $0<x<1$ and $\sigma$ is increasing for $x\le \frac12$. Since $s\le\frac{n_k}{N}$ and $i-s\le l<\delta n_k$, choosing $n_k$ and $N$ large enough yields
$$
\begin{aligned}
\sharp\,(\mathcal{C}_1)
\le& C \frac{\delta n_k^3}{N} \exp\bigl\{h(n_k-l)+\sigma(\frac{s}{n_k})n_k+\sigma(\frac{i-s}{l})l\bigr\}\\
\le C &\frac{\delta n_k^3}{N} \exp\bigl\{n_k(h+\sigma(1/N)
+\delta\log 2)\bigr\}.
\end{aligned}
$$
It follows that
\begin{multline}\label{eq:C-2}
\log\Big(\sum_{B_{n_k}(x,\epsilon)\in {\mathcal G}_1} e^{-\lambda n_k}\Big)\le\text{const} + 3\log(n_k)
\\+ n_k\left( (\delta+\frac{l(\epsilon)}{N})\log \beta(\epsilon) + \sigma(1/N) + \delta \log 2 - (\lambda-h)\right).
\end{multline}
Since $\delta$, $l(\epsilon)/N$ and $\sigma(1/N)$ can be chosen arbitrarily small, the dominant term in (\ref{eq:C-2}) is the term $-(\lambda-h)n_k$. This implies that the sum
$$
\sum_{B_{n_k}(x,\epsilon)\in{\mathcal G}_1}e^{-\lambda n_k}$$
is uniformly bounded if $\lambda>h$ and $n_k$ is large enough. We will make this precise later.

We now deal with the sum over $\mathcal{G}_2$. For an element 
$(\xi, \eta)\in {\mathcal C}_2$, let $m=m_{s-1}+\tau_{s-1}$. For a fixed 
$m$, there are infinitely many choices for $\tau_s$, and for each $\tau_s$ there are $e^{h\tau_s}$ choices of elements $J\in S$ with 
$\tau(J)=\tau_s$. On the other hand, similar to the case of 
$\mathcal{C}_1$, the number of different $V(\xi,\eta)$ for which
$m=m_{s-1} + \tau_{s-1}$ is bounded by
\begin{multline*}
\sum_{i=1}^{n_k}
\sum_{l=1}^{\delta n_k}
\sum_{s=1}^{m/N}
\binom{\delta n_k}{i-s}
\binom{m}{s}
e^{h(m_s-l)} \\
\le C \frac{\delta n_k^2m}{N}  \exp \bigl\{ \delta n_k\log 2+ (\sigma(1/N)+ h)m_s \bigr\}.
\end{multline*}
Combining the above estimate with (\ref{eq:C1-mult}) and (L2) and summing over all possible $m_s$ keeping in mind that 
$n_k\le m_s+\tau_s$ yields
\begin{multline*}
\sum_{B_{n_k}(x,\epsilon)\in{\mathcal G}_2}e^{-\lambda n_k} \\
\le\sum_{m_s=1}^{n_k}\sum_{\tau_s=\max\{N+1,n_k-m_s\}}^\infty C \frac{\delta n_k^2m_s}{N}  e^{\delta n_k \log 2+ (\sigma(1/N)+ h)m_s} e^{h\tau_s}e^{-(m_s+\tau_s)\lambda}\\
\times\beta(\epsilon)^{n_k\left(\delta +\frac{l(\epsilon)}{N}\right)}.
\end{multline*}
Assuming $\lambda>h$ and reordering terms yields
\begin{equation*}
\begin{aligned}
& \sum_{m_s=1}^{n_k}\frac{C\delta n_k^2m_s}{N}e^{\delta n_k\log 2+\sigma(1/N)m_s} e^{-(\lambda-h)m_s} \left(\sum_{\tau_s=\max\{N+1,n_k-m_s\}}^\infty  e^{-(\lambda-h)\tau_s}\right)\\
&\le \sum_{m_s=1}^{n_k}\frac{C\delta n_k^2m_s}{N} \, e^{\delta n_k\log 2+\sigma(1/N) m_s}e^{-(\lambda-h)m_s} \frac{1}{\lambda-h}e^{-(\lambda-h)(n_k-m_s)}\\
& \le \frac{C \delta n_k^3}{N}e^{\delta n_k\log 2}\frac{1}{\lambda-h} e^{-(\lambda-h)n_k}\frac{e^{\sigma(1/N)(n_k+1)}}{e^{\sigma(1/N)}-1} .
\end{aligned}
\end{equation*}
It follows that
\begin{equation}\label{eq:C-1}
\begin{aligned}
\log&\Big(\sum_{B_{n_k}(x,\epsilon)\in {\mathcal G}_2}e^{-\lambda n_k}\Big)\\
&\le \text{const}+ 3\log n_k -\log(\lambda-h) - \log(e^{\sigma(1/N)}-1)
 \\
&+ n_k \left( (\delta + \frac{l(\epsilon)}{N})\log \beta(\epsilon) + \sigma(1/N) + \delta\log 2 -(\lambda-h) \right).
\end{aligned}
\end{equation}
Since $\delta$, $l(\epsilon)/N$ and $\sigma(1/N)$ can be chosen arbitrarily small, the dominant term is $-n_k(\lambda-h)$ hence
$$
\sum_{B_{n_k}(x,\epsilon)\in{\mathcal G}_2} e^{-\lambda n_k}
$$
is uniformly bounded for sufficiently large $n_k$. Now 
$M(Z,\lambda,\epsilon,n_k)$ is estimated as follows: given $\epsilon>0$, choose $\lambda>h$ and then $N>N(\epsilon)$ (which depends on $l(\epsilon)$) so large that both $\frac{l(\epsilon)}{N}$ and $\sigma(1/N)$ are small enough to ensure
$$
\frac{l(\epsilon)}{N}\log\beta(\epsilon) + \sigma(1/N) <\frac13(\lambda-h).
$$
Pick $\delta>0$ so small that
$$
\delta\log \beta(\epsilon)+ \delta\log 2 <\frac13(\lambda-h).
$$
With $\epsilon, N, \delta$ fixed, for each large enough $n_k$ consider the collection $\mathcal{C}(\epsilon, N, \delta, n_k)$ and use it to estimate $M(Z,\lambda, \epsilon, n_k)$ as above. Moreover, notice that the terms that are linear in $n_k$ in both (\ref{eq:C-2}) and (\ref{eq:C-1}) are bounded by $-\frac13(\lambda-h)n_k$ and hence as $n_k\to \infty$ both (\ref{eq:C-2}) and (\ref{eq:C-1}) approach $-\infty$. It follows that 
$m(Z,\lambda,\epsilon)< 0$ for any $\lambda>h$. Hence, 
$h_{\text{top}}(f, Z)\le h$ and the desired result follows.
\end{proof}
By Lemmas~\ref{entropybound1} and \ref{entropybound2} Proposition~\ref{smallentropy} holds. This completes the proof of Theorem~\ref{liftability}.
\end{proof}

The following statement is an immediate corollary of Theorems~\ref{liftability} and \ref{liftgibbs}.

\begin{thm}\label{h-equilibrium} Let $f$ be a continuous map of a compact metric space $X$ with $h_{\text{top}}(f)<\infty$ and $\varphi$ a potential function. Assume that $f$ admits an inducing scheme 
$\{S,\tau\}$ of hyperbolic type satisfying Conditions (I3) and (I4). Assume also that $\varphi$ satisfies the conditions of Theorem~\ref{liftgibbs}. 
Let $\mu_\varphi$ be the measure constructed in Theorem~\ref{liftgibbs} and let $h$ be the constant $h$ in Condition (L2).

If $h_{\mu_\varphi}(f)>h$, then $\mu_\varphi$ is the unique equilibrium measure for $\varphi$ among all ergodic invariant measures
$\mu$ with $\mu(Y)=1$ (see \eqref{setY} for the definition of the set $Y$) and $h_\mu(f)>h$. In particular, this holds if $\varphi$ satisfies Conditions (P1)-(P4).
\end{thm}
\begin{rmk}
\end{rmk}
In the one dimensional setting inducing schemes satisfying Condition (L1) are often constructed for maps with a set of critical/singular points 
$\mathcal{C}$ satisfying the natural assumption that there exists a critical/singular point $c^*\in \mathcal{C}$ whose set of pre-images 
$\{f^{-n}(c^*)\}_{n\in\N}$ is dense and does not intersect $\mathcal{C}$ (see Condition (H3) in \cite{DiaHolLuz06}). The inducing scheme can then be constructed by taking as the inducing domain a small enough neighborhood $(c^*-\delta^*, c^*+\delta^*)$. Condition (L1) can then be shown to hold as in \cite[Lemma 1]{DiaHolLuz06}. More precisely, by assumption, for every $\delta>0$ there exists $N>0$ such that the set  
$\{f^{-n}(c^*)\}_{n=0}^{N}$ is $\delta$-dense and uniformly bounded away from the set $\mathcal{C}$. One can choose $\delta$ so small that for any interval $I$ with $|I|\ge\varepsilon$ there exists $f^{-k}(c^*)$ with $k<N$ close to the center of $I$. Thus, by choosing $\delta^*$ small enough, one can ensure that any $f^n(I)$ with $|f^n(I)|\ge\varepsilon$ contains a diffeomorphic pre-image with bounded distortion of the inducing domain. Condition (L1) follows by contrapositive: assume $|f^n(J)|\ge \varepsilon$ for some $J\in S$ and $n<\tau(J)-N$. Then the above argument shows that $c^*\in f^{i+n}(J)$ for some $n< N$ contradicting the fact that $f^{\tau(J)}$ is a diffeomorphism. 

In general in the higher dimensional setting this argument cannot be carried out as one lacks a proper identification of the critical/singular set $\mathcal{C}$.

\section{Inducing Schemes of Hyperbolic Type for Young's Diffeomorphisms}\label{sec:tower}

Theorem~\ref{equilibrium1} provides a way to construct equilibrium measures for maps with inducing schemes of hyperbolic type. Here we show that maps introduced by L.-S. Young \cite{You98} admit inducing schemes of hyperbolic type that satisfy our conditions.

\subsection{Young's diffeomorphisms.} Consider a $C^{1+\epsilon}$ diffeomorphism $f:M\to M$ of a compact smooth Riemannian manifold $M$. Following \cite{You98} we describe a collection of conditions on the map $f$.

An embedded disk $\gamma\subset M$ is called an \emph{unstable disk} (respectively, a \emph{stable disk}) if for all $x,y\in\gamma$ we have that
$d(f^{-n}(x), f^{-n}(y))\to 0$ (respectively, $d(f^{n}(x), f^{n}(y))\to 0$) as
$n\to+\infty$. A collection of embedded disks $\Gamma^u=\{\gamma^u\}$\label{sb:gamma} is called a \emph{continuous family of $C^1$ unstable disks} if there exists a homeomorphism
$\Phi: K^s\times D^u\to\cup\gamma^u$ satisfying:
\begin{itemize}
\item $K^s\subset M$ is a Borel subset and $D^u \subset \R^d$ is the closed unit disk for some $d<dim M$;
\item $x\to\Phi|{\{x\}\times D^u}$ is a continuous map from $K^s$ to the space of $C^1$ embeddings of $D^u$ into $M$ which can be extended to a continuous map of the closure $\overline{K^s}$;
\item $\gamma^u=\Phi(\{x\}\times D^u)$ is an unstable disk.
\end{itemize}
A \emph{continuous family of $C^1$ stable disks} is defined similarly.

We allow the sets $K^s$ to be non-compact in order to deal with overlaps which appear in most known examples including those in Sections \ref{Henon} and \ref{katok-map}.

A set $\Lambda\subset M$ has \emph{hyperbolic product structure} if there exists a continuous family $\Gamma^u=\{\gamma^u\}$ of unstable disks $\gamma^u$ and a continuous family 
$\Gamma^s=\{\gamma^s\}$ of stable disks $\gamma^s$ such that
\begin{itemize}
\item $\text{dim }\gamma^s+\text{dim }\gamma^u=\text{dim } M$;
\item the $\gamma^u$-disks are transversal to
$\gamma^s$-disks by an angle uniformly bounded away from $0$;
\item each $\gamma^u$-disks intersects each 
$\gamma^s$-disk at exactly one point;
\item $\Lambda=(\cup\gamma^u)\cap(\cup\gamma^s)$.
\end{itemize}

A subset $\Lambda_0\subset\Lambda$ is called an \emph{$s$-subset} if it has hyperbolic product structure and is defined by the same family 
$\Gamma^u$ of unstable disks as $\Lambda$ and a continuous subfamily $\Gamma_0^s\subset\Gamma^s$ of stable disks. A 
\emph{$u$-subset} is defined analogously.

For an $s$-subset $\Lambda_0\subset\Lambda$ and a $u$-disk $\gamma^u\in\Gamma^u$ the set $\Lambda_0\cap\gamma^u$ may not be compact. 
By the definition of the family $\Gamma^u$, for any $x$ in the closure $\overline{\Lambda_0\cap\gamma^u}$ there exists a unique $s$-disk $\gamma^s(x)\in\Gamma^s$. We define the \emph{$s$-closure} $scl(\Lambda_0)$ of $\Lambda_0$ by 
$$
scl(\Lambda_0):=\bigcup_{x\in\overline{\Lambda_0\cap\gamma^u}}\gamma^s(x)\cap\Lambda.$$
It is easy to see that the $s$-closure does not depend on the choice of $\gamma^u$ and that it is an $s$-subset over a continuous family of stable disks which we denote by $\overline{\Gamma_0^s}$. We define the $u$-closure $ucl(\Lambda_1)$ of a given $u$-subset 
$\Lambda_1\subset\Lambda$ similarly:
$$
ucl(\Lambda_1):=\bigcup_{x\in\overline{\Lambda_1\cap\gamma^s}}\gamma^u(x)\cap\Lambda.$$
Assume the map $f$ satisfies the following conditions:
\begin{enumerate}
\item[(Y0)] There exists $\Lambda\subset M$ with hyperbolic product structure such that $\mu_{\gamma^u}(\gamma^u\cap \Lambda)>0$ for all $\gamma^u\in \Gamma^u$, where $\mu_{\gamma^u}$ is the leaf volume on $\gamma^u$. 
\item[(Y1)] There exists a countable collection of continuous subfamilies $\Gamma_i^s\subset\Gamma^s$ of stable disks and positive integers $\tau_i$, $i\in\N$ such that the $s$-subsets 
\begin{equation}\label{Lambda}
\Lambda_i^s:=\bigcup_{\gamma\in\Gamma^s_i}\,\bigl(\gamma\cap \Lambda\bigr)\subset\Lambda
\end{equation}
are pairwise disjoint and satisfy
\begin{enumerate}
\item \emph{invariance}: for every $x\in\Lambda_i^s$ 
$$
f^{\tau_i}(\gamma^s(x))\subset\gamma^{s}(f^{\tau_i}(x)), \,\, f^{\tau_i}(\gamma^u(x))\supset\gamma^u(f^{\tau_i}(x)),
$$
where $\gamma^{u,s}(x)$ denotes the (un)stable disk containing $x$.
\item \emph{Markov property}: $\Lambda_i^u:=f^{\tau_i}(\Lambda_i^s)$ is a $u$-subset of $\Lambda$ such that for all
$x\in\Lambda_i^s$ 
$$
\begin{aligned}
f^{-\tau_i}(\gamma^s(f^{\tau_i}(x))\cap\Lambda_i^u)
&=\gamma^s(x)\cap \Lambda,\\
f^{\tau_i}(\gamma^u(x)\cap\Lambda_i^s)
&=\gamma^u(f^{\tau_i}(x))\cap \Lambda.
\end{aligned}
$$
\end{enumerate}
\end{enumerate}
It is easy to see that Conditions (Y1a) and (Y1b) also hold for 
$x\in scl(\Lambda_i^s)$ after replacing $\Lambda$ by its $s$-closure.
\vspace*{.5cm}
\begin{enumerate}
\item[(Y2)] For every $\gamma^u\in\Gamma^u$ one has
$$
\mu_{\gamma^u}\left((\Lambda\setminus\cup\Lambda_i^s)\cap\gamma^u\right)=0\qquad\text{ and}\qquad \mu_{\gamma^u}\left((scl(\Lambda_i^s)\setminus\Lambda_i^s)\cap\gamma^u\right)=0.
$$
\end{enumerate}
For any $x\in \Lambda^s_i$ define the \emph{inducing time} by 
$\tau(x):=\tau_i$ and the \emph{induced map} 
$F: \bigcup_{i\in\N}\Lambda_i^s\to\Lambda$
by
$$
F|_{\Lambda_i^s}:=f^{\tau_i}|_{\Lambda_i^s}.$$
\begin{enumerate}
\item[(Y3)] There exists $0<\alpha<1$ such that for any 
$i\in\N$ we have:
\begin{enumerate}
\item[(a)] For $x\in\Lambda_i^s$ and $y\in\gamma^s(x)$,
$$
d(F(x), F(y))\le\alpha\, d(x,y);
$$
\item[(b)] For $x\in\Lambda^s_i$ and  
$y\in\gamma^u(x)\cap \Lambda_i^s$,
$$
d(x,y)\le\alpha \, d(F(x), F(y)).
$$
\end{enumerate}
\end{enumerate}
For $x\in \Lambda$ denote 
$$
E^u(f^k(x)):=T_{f^k(x)}f^k(\gamma^u(x))=Df^k T_{x}\gamma^u(x)
$$ and denote by 
$\J f(x):=|\det Df|_{E^u(x)}|$ the restriction of the Jacobian of $f$  to $E^u$ (the definition of $\J F(x)$ is analogous).
\begin{enumerate}
\item[(Y4)] There exist $c>0$ and $0<\beta<1$ such that:
\begin{enumerate}
\item[(a)] For all $n\ge 0$, $x\in F^{-n}(\cup_{i\in\N}\Lambda^s_i)$ 
and $y\in\gamma^s(x)$ we have
\[
\left|\log\frac{\J F(F^{n}x)}{\J F(F^{n}y)}\right|\le c\beta^n;
\]
\item[(b)] For any $i_0,\dots, i_n\in\inte$, 
$F^k(x),F^k(y)\in\Lambda^s_{i_k}$ for $0\le k\le n$ and  
$y\in\gamma^u(x)$ we have
\[
\left|\log
\frac{ \J F(F^{n-k}x)}{\J F(F^{n-k}y)}\right|\le c\beta^k.
\]
\end{enumerate}
\end{enumerate}
Observe that the induced map $F$ can be extended to the $s$-closure $scl(\Lambda_i^s)$ in such a way that Conditions (Y3) and (Y4) hold with the same constants. 
\begin{enumerate}
\item[(Y5)] There exists $\gamma^u\in\Gamma^u$ such that 
$$
\sum_{i=1}^\infty \tau_i \mu_{\gamma^u}(\Lambda_i^s) <\infty.
$$
\end{enumerate}
For any $\gamma, \gamma'\in\Gamma^u$ denote by $\Theta$ the holonomy map (along the stable leaves) associated with $\gamma, \gamma'$:
$$\begin{array}{rl}
\Theta:\gamma\cap\Lambda&\to\gamma'\cap\Lambda\\
x&\mapsto\gamma^s\cap\gamma'.
\end{array}$$
Conditions (Y1), (Y3) and (Y4) and standard arguments (see for example, \cite{BarPes07}) yield the following:
\begin{prop}\label{Jacobian}
For $x, y\in\bigcap_{n=0}^{\infty}F^{-n}(\cup_{i\in\N}\Lambda^s_i)$
the holonomy map $\Theta$ associated with 
$\gamma=\gamma^u(x),\gamma'= \gamma^u(y)$ is absolutely continuous and its Jacobian is given by the formula
$$
\frac{d(\Theta_*^{-1}\mu_{\gamma'})}{d\mu_\gamma}(x)=\prod_{n=0}^{\infty}\frac{\J F(F^nx)}
{\J F(F^n(\Theta (x)))}
$$  
and is uniformly bounded from above and below. 
\end{prop}

\subsection{Construction of the inducing scheme.} 

Let $f$ be a $C^{1+\epsilon}$ diffeomorphism satisfying conditions (Y1) and (Y3). Set
\begin{equation}\label{W}
W:=\bigcap_{n=-\infty}^{\infty}ucl\Bigl(F^n\bigl(\bigcup_{i\in\N}\Lambda_i^s\bigr)\Bigr)
\end{equation}
be the non-empty maximal $F$-invariant set contained in $ucl(\Lambda)$. 
Set 
\begin{equation}\label{insc}
S=\{\Lambda_i^s\cap W\}_{i\in\N}\quad\text{ and }\quad \tau(\Lambda_i^s\cap W)=\tau_i.
\end{equation}
For any two-sided sequence $a=(\ldots, a_{-1},a_0,a_1,\ldots)\in S^\Z$, $n\in\N$ and denote 
\begin{equation}\label{cylinder}
\begin{aligned}
\Lambda_{a_0,\ldots,a_n}:&=\{x\in\Lambda^s_{a_0}\colon F^j(x)\in\Lambda^s_{a_j}\quad 0\le j\le n\}\\
\Lambda_{a_{-n} \ldots,a_0}:&=\{x\in\Lambda^s_{a_0}\colon F^{-j}(x)\in\Lambda^s_{a_{-j}}\quad 0\le j\le n\}
\end{aligned}
\end{equation}
By Condition (Y1), the sequence $\{\Lambda_{a_0,\ldots,a_n}\}_{n\in\N}$ is a nested sequence of $s$-subsets and $\{\Lambda_{a_{-n} \ldots,a_0}\}_{n\in\N}$ is a nested sequence of $u$-subsets and the same holds for their $u$-closures. 
\begin{prop}\label{isyd}
The pair $\{S, \tau\}$ is an inducing scheme of hyperbolic type for $f$ satisfying Condition (I4). If, in addition, the sets $scl(\Lambda_i^s)$ are pairwise disjoint, then $\{S, \tau\}$ also satisfies Conditions (I3).
\end{prop}
\begin{proof}
$\{S, \tau\}$ is an inducing scheme of hyperbolic type for $f$ provided it satisfies Conditions (I1) and (I2). Condition (I1) follows from Condition (Y1) and the $F$-invariance of $W$. We now prove Condition (I2). By the hyperbolic product structure of $\Lambda$, the intersection 
$\cap_{n=0}^\infty M_n\neq \emptyset$
where
\begin{equation}\label{mw}
M_n:=ucl(\Lambda_{a_{-n} \ldots,a_0})\cap scl(\Lambda_{a_0 \ldots,a_n}).
\end{equation}
To prove that this intersection consists of a single point we show that 
$\displaystyle{\lim_{n\to\infty}\text{diam}M_n=0}$. 
Indeed, for any $x,y\in M_n$ let  $z=\gamma^s(x)\cap\gamma^u(y)$. By Condition (Y1), we have $F^k(z)=\gamma^s(F^kx)\cap\gamma^u(F^ky)$ for
$-n\le k \le n$.  By Condition (Y3), 
$$
d(z,y)=d(F^{-n}(F^nz),F^{-n}(F^ny))\le c\alpha^n
$$
and 
$$
d(x,z)=d(F^n(F^{-n}x),F^n(F^{-n}z))\le c\alpha^n
$$
for some constant $c>0$. Then $d(x,y)\le d(x,z)+d(z,y)\le 2c\alpha^n$ and Condition (I2) follows.

Condition (I4) follows from the fact that the coding map $\pi$ is a (semi-)conjugacy (see Proposition~\ref{conjugacy}). 

If the sets $scl(\Lambda_i^s)$ are pairwise disjoint then  $S^\Z\setminus\check{S}=\emptyset$ and Condition (I3) follows.
\end{proof}

\subsection{Liftability problem}
 
In general, solving the liftability problem requires verification of Conditions (L1) and (L2). In our examples of Young's diffeomorphisms in Sections~\ref{katok-map} and \ref{Henon} the inducing time is the first return time and hence, any measure charging the base of the tower is liftable. Nevertheless, a strengthened version of Condition (L2) (see \eqref{SN}) is used to establish the thermodynamical formalism (see
Theorem~\ref{geom_poten}).

\section{Thermodynamics of Young's diffeomorphisms with the geometric potential}

For $t\in\mathbb{R}$ consider the family of potential functions
$$
\varphi_t(x):=-t\log\J f(x).
$$
\begin{thm}\label{geom_poten}
Let $f\colon M\to M$ be a $C^{1+\epsilon}$ diffeomorphism of a compact smooth Riemannian manifold $M$ satisfying Conditions (Y0)-(Y5). Then the following statements hold:
\begin{enumerate}
\item There exists an equilibrium measure $\mu_1$ for the potential 
$\varphi_1$ which is a unique SRB measure;
\item Assume that the inducing scheme $\{S,\tau\}$ for $f$ satisfies Conditions (I3) and (I4) and the following condition:
\begin{equation}\label{SN}
S_n=\sharp\,\{J\in S\colon \tau(J)=n\}\le e^{hn}
\end{equation}
with $0<h <-\int\varphi_1\,d\mu_1$. Then there exists $t_0<0$ such that for every $t_0<t< 1$ there exists a measure $\mu_t\in\mathcal{M}(f,Y)$ which is a unique equilibrium measure for the potential $\varphi_t$ among measures in $\mathcal{M}_L(f, Y)$, i.e.,
$$
h_{\mu_t}(f)+\int_Y\varphi_t\,d\mu_t 
=\sup_{\mu\in\mathcal{M}_L(f,Y)}\{h_{\mu}(f)+\int_Y\varphi_t\,d\mu\},
$$
where $Y$ is defined in \eqref{setY}.
\item If, in addition, the inducing scheme $\{S,\tau\}$ satisfies Condition (L1), then $\mu_t$ is the unique equilibrium measure for $\varphi_t$ among all measures $\mu$ with $h_\mu(f)>h$ i.e.,
$$
h_{\mu_t}(f)+\int_Y\varphi_t\,d\mu_t 
=\sup_{\stackrel{\mu\in\mathcal{M}(f,Y)}{h_\mu(f)>h}}\{ h_{\mu}(f)+\int_Y\varphi_t \,d\mu\}.
$$
\end{enumerate}
\end{thm}
The proof of this theorem occupies the rest of this section. 
\begin{proof}
In view of Theorem~\ref{equilibrium1} it is sufficient to prove that the potential function $\varphi_t$ satisfies conditions (P1)--(P4) for 
$t_0<t< 1$ with some $t_0<0$. Since 
$\overline{\varphi}_t(x)=-t\log\J F(x)$, Condition (P1) is straightforward and in fact, holds for $\varphi_t+c_t$ for any constant $c_t\in\R$. Lemma~\ref{localholder} below implies that $\varphi_t+c_t$ satisfies Condition (P2) for any constant $c_t\in\R$. It follows from Lemma~\ref{FGP} below that for all $t\in\mathbb{R}$ there exists some $c_t\in\R$ such that $\varphi_t+c_t$ satisfies Condition (P3). The existence of the measure $\mu_1$ in Statement (1) is a well known result (see \cite{You98}). We prove it in view of our construction in Lemma~\ref{pres-est}. The fact that $\mu_1$ is an SRB measure follows from \cite{LedStr82}. Lemma~\ref{interval} shows that under Condition~\eqref{SN} there exists some $t_0<0$ such that $\varphi_t$ (and hence $\varphi_t+c_t$) with $t_0<t<1$ satisfies Condition (P4). Since $\varphi_t$ and $\varphi_t+c_t$  are cohomologous they admit the same equilibrium measures. Statement (2) follows from Theorem~\ref{equilibrium1}. Statement (3) then follows from Theorem~\ref{h-equilibrium}. This completes the proof of Theorem \ref{geom_poten}.
\end{proof}

\begin{lem}\label{cylindersize}
There are constants $C_1>0$ and $\lambda_1>1$ such that for any 
$(a_0,\dots,a_{n-1})\in S^n$, any 
$x\in scl(\Lambda_{a_0, \dots, a_{n-1}})$ (see \eqref{cylinder}), and 
$\gamma\in \Gamma^u$,
$$
\begin{aligned}
C_1^{-1}(\J (F^n)(x))^{-1}\le&
\mu_{\gamma}(\gamma\cap\Lambda_{a_0, \dots,a_{n-1}})\\
=&
\mu_{\gamma}(\gamma\cap scl(\Lambda_{a_0, \dots,a_{n-1}}))\le C_1(\J (F^n)(x))^{-1},
\end{aligned}
$$
$$
\J (F^n)(x)\le \lambda_1^{\tau_{a_0}+\cdots+\tau_{a_{n-1}}}.
$$
\end{lem}
\begin{proof}
Observe that $F^n$ maps $\gamma\cap\Lambda_{a_0, \dots,a_{n-1}}$ homeomorphically onto $\gamma'\cap \Lambda$. Hence, 
$$
\mu_{\gamma'}(\gamma'\cap\Lambda)=\int_{\gamma\cap \Lambda_{a_0, \dots,a_{n-1}}}\J (F^n)(x) \,d\mu_\gamma(x).
$$
By Condition (Y0) and Proposition~\ref{Jacobian}, there exists a constant $A>0$ (which does not depend on the choice of $\gamma'$) such that
$$
A^{-1}\le\frac{\mu_{\gamma'}(\gamma'\cap\Lambda)}{\mu_{\gamma}(\gamma\cap\Lambda)}\le A.
$$
For $x,y\in \gamma\cap scl(\Lambda_i^s)$ Condition (Y4)(b) yields that
$$
B^{-1}\le \frac{\J (F^n)(x)}{\J (F^n)(y)}\le B
$$
for some $B>0$. This together with Condition (Y2) proves the first part. The last inequality follows by setting   $\lambda_1:=\max_{x\in\overline{\Lambda}}\J f(x)$ and observing that this maximum is finite due to the continuity of the family of unstable disks over the compact set $\overline{K^s}$.
\end{proof}
\begin{lem}\label{localholder}
The potential $\varphi_t$  satisfies Condition $(P2)$ for all $t\in {\mathbb R}$.
\end{lem}
\begin{proof}
To prove Condition (P2) consider 
$(a_{-n+1},\dots, a_{n-1})\in S^{2n-1}$ and $x,y \in M_n$ (see \eqref{mw}). If $z=\gamma^s(x)\cap\gamma^u(y)$ then 
$F^{-n}(z)\in\gamma^s(F^{-n}(x))$ and Condition (Y4)(a) yields
\begin{multline*}
|\bar\varphi_t(x)-\bar\varphi_t(z)| =|t|
\left|\log\frac{\J F(x)}{\J F(z)} \right|
=|t|\left|\log\frac{\J F(F^n(F^{-n}x))}{\J F(F^n(F^{-n}z))} \right| \le c|t|\beta^n.
\end{multline*}
A similar argument that uses (Y4)(b) instead of (Y4)(a) yields
$$
|\bar\varphi_t(y)-\bar\varphi_t(z)|\le c|t|\beta^n
$$
and the desired result follows since $x$ and $y$ were chosen arbitrarily.
\end{proof}

\begin{lem}\label{FGP}
The potential function $\varphi_t+c$ satisfies Condition $(P3)$ for all 
$c\le 0$ and all $t\ge 1$. Additionally, if Condition~\eqref{SN} is satisfied, then for all $t\in {\mathbb R}$ there exists $c_t\in\R$ such that 
$\varphi_t+c$ satisfies Condition $(P3)$ for all $c< c_t$.
\end{lem}
\begin{proof}
It suffices to prove the existence of some $c_t\in\R$ such that for all 
$c< c_t$ one has
$$
\sum_{i=1}^\infty\sup_{x\in\overline{\Lambda_i^s}}|
\J F(x)|^{-t}e^{c\tau_i}<\infty.
$$
For $t\ge 1$ and all $c\le c_t=0$ Lemma~\ref{cylindersize} and (Y5) yield
$$
\sum_{i=1}^\infty\sup_{x\in scl(\Lambda_i^s)}(\J F(x))^{-t}e^{c\tau_i}
\le C_1\sum_{i=1}^\infty\mu_\gamma(\gamma\cap\Lambda_i^s)^t
\le C_1\sum_{i=1}^\infty\mu_\gamma(\gamma\cap\Lambda_i^s)
<\infty.
$$
For $0\le t<1$ Lemma~\ref{cylindersize} and \eqref{SN} yield for any $c<c_t=-h$ 
$$
\sum_{i=1}^\infty\sup_{x\in scl(\Lambda_i^s)}(\J F(x))^{-t}e^{c\tau_i}
\le\mu_\gamma(\gamma\cap \Lambda)^t \sum_{i=1}^\infty e^{c\tau_i}
\le\mu_\gamma(\gamma\cap \Lambda)^t \sum_{n=1}^\infty e^{(c +h )n}
<\infty.
$$
For $t\le 0$ and $c_t=t\log\lambda_1-h$ Lemma~\ref{cylindersize}  and \eqref{SN} yields
$$
\sum_{i=1}^\infty\sup_{x\in scl(\Lambda_i^s)}(\J F(x))^{-t}e^{c\tau_i} \le
\sum_{n=1}^\infty e^{(c-t\log\lambda_1 +h )n}
<\infty.
$$
\end{proof}

\begin{lem}\label{pres-est}
Let $P_L(t):=P_L(\varphi_t)$. Then $P_L(1)=0$, and there exists a unique measure $\mu_1$ such that
$$
h_{\mu_1}(f)+\int_Y\varphi_1 \,d\mu_1 = P_L(1).
$$
Moreover, for all $t\in\R$ one has
$$
P_L(t) \ge (t-1)\int_Y\varphi_1\,d\mu_1.
$$
\end{lem}
\begin{proof}[Proof of the lemma]
Let $\Phi_1:=\bar\varphi_1\circ\pi$, where the coding map $\pi$ is defined in \eqref{def:pi}. For any $b\in S$ Lemma~\ref{cylindersize} yields 
\begin{eqnarray*}
P_G(\Phi_1) &=&\lim_{n\to\infty}\frac1n\log\sum_{\sigma^n(\underline{a})=\underline{a}}
\exp \sum_{i=0}^{n-1} \Phi_1^i(\underline{a}) \mathbb{I}_{[b]}(\underline{a}) \\
&=& 
\lim_{n\to\infty} \frac1n \log\sum_{F^n(x)=x\in scl(\Lambda_b^s)}(\J (F^n)(x))^{-1}\\
&=&
\lim_{n\to\infty} \frac1n\log\sum_{\substack{a_0, \cdots, a_{n}\in S \\ a_0=a_{n}=b}} \mu_{\gamma}(\gamma\cap scl(\Lambda_{a_0, \dots,a_{n-1}})) \\
&=& 
\lim_{n\to \infty}\frac1n \log \mu_\gamma(\gamma\cap scl(\Lambda_b^s))=0.
\end{eqnarray*}
The last line follows from the fact that the sum is taken over all possible cylinders $\Lambda_{b, a_1, \dots,a_{n-1}}$ of length $n$ contained in 
$\Lambda_b$. For every $\mu\in\mathcal{M}_L(f, X)$ we have $Q_{\mathcal{L}^{-1}(\mu)}<\infty$, so Proposition~\ref{abramovkac} yields $P_L(1)\le P_G(\Phi_1)= 0$.

To prove the lower bound, note that  Lemmas~\ref{localholder} and \ref{FGP} (with constant $c_1=0$) imply that  $\varphi_1$ satisfies Conditions (P2) and (P3). Since $\varphi_1$ clearly satisfies Condition (P1), Theorems~\ref{conditions} and \ref{gibbs2} imply the existence of a unique Gibbs measure $\nu_{\Phi_1}$ for $\Phi_1$. Then by Relation \eqref{eq:gibbs2}, the measure 
$\nu_{\bar{\varphi}_1}:=\pi_*\nu_{\Phi_1}$ satisfies the corresponding Gibbs property, i.e., there is $C_0>$ such that for any $n$-tuple $(b_{0}, \dots, b_{n-1})$ and any
$x\in\Lambda^s_{b_{0}\dots b_{n-1}}$ we have
$$
C_0^{-1}\le \frac{\nu_{\bar{\varphi}_1}(\Lambda^s_{b_{0}\dots b_{n-1}})}{\exp(-nP_G(\Phi_1)+(\bar{\varphi}_1)_n(x))} \le C_0.
$$
Setting $n=1$ by Lemma~\ref{cylindersize}), we find that 
$$
\nu_{\bar{\varphi}_1}(\Lambda_i^s)\le C_0 (\J F(x))^{-1}.
$$
Now Condition (Y5) yields
$$
Q_{\nu_{\bar{\varphi}_1}}=\sum_{i\in\N} \tau_i\nu_{\bar{\varphi}_1}(\Lambda_i^s)\le C_0 \sum_{i\in\N}\tau_i(\J F(x))^{-1}<\infty.
$$
By Proposition~\ref{abramovkac}, $h_{\nu_{\bar{\varphi}_1}}(F)<\infty$ and 
Theorem~\ref{gibbs2} (3) yields
$$
h_{\nu_{\bar{\varphi}_1}}(F) + \int_W \bar\varphi_1 \,d\nu_{\bar{\varphi}_1} = P_G(\Phi_1)=0.
$$
Furthermore,
\[
P_L(1)\ge h_{{\mathcal L}(\nu_{\bar{\varphi}_1})}(f) + 
\int_Y\varphi_1 \,d {\mathcal L}(\nu_{\bar{\varphi}_1}) = \frac{h_{\nu_{\bar{\varphi}_1}}(F)+\int_W \bar\varphi_1 \,d\nu_{\bar{\varphi}_1}}{Q_{\nu_{\bar{\varphi}_1}}} =0.
\]
Therefore $P_L(1)=0$. Hence $\Phi^+_1=\Phi_1$ and  $\nu_{\varphi^+_1}=\nu_{\bar\varphi_1}$. Statement (1) now follows from Theorem~\ref{equilibrium1} for $\mu_1=\mathcal{L}(\nu_{\bar\varphi_1})$.

To prove the second statement note that
\begin{multline*}
P_L(t) = \sup_{\mu\in \mathcal{M}_{L}(f, Y)}\{ h_\mu(f)+\int_Y \varphi_t\,d\mu\} \ge  h_{\mu_1}(f)+\int_Y\varphi_t\,d\mu_1\\
=h_{\mu_1}(f)+t\int_Y\varphi_1\,d\mu_1 
=(t-1)\int_Y\varphi_1\,d\mu_1
\end{multline*}
and the desired result follows.
\end{proof}
\begin{lem}\label{interval} If Condition~\eqref{SN} is satisfied, the potential $\varphi_t$ satisfies Condition $(P4)$ for all $t_0<t<1$ where 
$$
t_0:=\frac{h+\int_X\varphi_1\,d\mu_1}{\log\lambda_1+\int_X\varphi_1\,d\mu_1}<0.
$$
\end{lem}
\begin{proof}
By Lemma~\ref{cylindersize}
$$
\sum_{i=1}^\infty \tau_i\sup_{x\in \Lambda_i^s}(\J F(x))^{-t} e^{-(P_L(t)-\varepsilon)\tau_i}
\le C\sum_{n=1}^\infty ne^{-(P_L(t)-\varepsilon)n}\sum_{\tau_i=n} \mu_\gamma(\gamma\cap\Lambda_i^s)^t
$$
for some constant $C>0$, so it suffices to prove that the right hand side is finite. 

For $0\le t<1$, Jensen's inequality together with Lemma~\ref{pres-est} and Condition~\eqref{SN} yield:
\begin{multline*} 
\sum_{n=1}^\infty ne^{-(P_L(t)-\varepsilon)n}\sum_{\tau_i=n} \mu_\gamma(\gamma\cap\Lambda_i^s)^t\\
\le \sum_{n=1}^\infty n e^{-(P_L(t)-\varepsilon)n} S_n^{1-t}\left( \sum_{\tau_i=n} \mu_\gamma(\gamma\cap\Lambda_i^s) \right)^t \\
\le  \left(\mu_{\gamma}(\gamma\cap\Lambda)\right)^t\sum_{n=1}^\infty n\  e^{(1-t)(h+\int_X\varphi_1\,d\mu_1)n+
\varepsilon n}<\infty
\end{multline*}
provided one chooses $0<\varepsilon<(1-t)(-\int_X\varphi_1\,d\mu_1-h)$ for $t<1$ given.
For $t<0$ Lemma~\ref{cylindersize} implies
\begin{multline*}
\sum_{n=1}^\infty ne^{-(P_L(t)-\varepsilon)n}\sum_{\tau_i=n} \mu_\gamma(\gamma\cap\Lambda_i^s)^t\le\sum_{n=1}^\infty ne^{-(P_L(t)-\varepsilon)n} S_n  \lambda_1^{n|t|} \\
\le C\sum_{n=1}^\infty n\lambda_1^{n|t|} e^{(h +\varepsilon +(1-t)\int_X\varphi_1\,d\mu_1)n} <\infty
\end{multline*}
if $h +\varepsilon+(1-t)\int_X\varphi_1\,d\mu_1< t\log\lambda_1$. Since $h<-\int_X\varphi_1\,d\mu_1\le \log\lambda_1$, this holds for
$\frac{h+\int_X\varphi_1\,d\mu_1}{\log\lambda_1+\int_X\varphi_1\,d\mu_1}<t<0$
and $\varepsilon$ sufficiently small. 
\end{proof}
We now describe some ergodic properties of the measure $\mu_t$.
\begin{thm}\label{geom_poten1}
Let $f:M\to M$ be a $C^{1+\epsilon}$ diffeomorphism of a compact smooth Riemannian manifold $M$ satisfying Conditions (Y0)-(Y5). Assume that the inducing scheme $\{S,\tau\}$ satisfies Conditions (I3), (I4) and Condition~\eqref{SN}. Then there exists $t_0<0$ such that 
\begin{enumerate}
\item for every $t_0<t\le 1$ the equilibrium measure $\mu_t$ is ergodic;
\item for every $t_0<t< 1$ $\mu_t$  has exponential decay of correlations and satisfies the CLT with respect to a class of potential functions which contains all H\"older continuous functions on $M$.
\end{enumerate}
\end{thm}
\begin{proof}
The proof follows from Theorems~\ref{equilibrium1}
and \ref{geom_poten} after checking that the measure $\mathcal{L}^{-1}(\mu_t)\in\mathcal{M}(F,W)$ has exponential tail. First observe that Lemma~\ref{interval} and Theorem~\ref{conditions} (3) yield $0=P_G(\Phi^+)=P_G(\Phi)-P_L(t)$. Since $\mathcal{L}^{-1}(\mu_t)$ is a Gibbs measure, the Gibbs property~\eqref{eq:gibbs2} and Condition (Y4) imply that for any $x\in scl(\Lambda_i^s)$ one has
$$
\mathcal{L}^{-1}(\mu_t)(\Lambda_i^s)\le C_0(\J F(x))^{-t}e^{-P_L(t)\tau_i}.
$$
By Lemma~\ref{cylindersize} there exists some $C>0$ such that
\begin{multline*}
\sum_{\tau_i\ge N}\tau_i\ \mathcal{L}^{-1}(\mu_t)(scl(\Lambda_i^s))\le C_0 \sum_{\tau_i\ge N} \tau_i\sup_{x\in scl(\Lambda_i^s)}(\J F(x))^{-t} e^{-P_L(t)\tau_i}\\  
\le C\sum_{n\ge N} ne^{-P_L(t) n}\sum_{\tau_i=n} \mu_\gamma(\gamma\cap\Lambda_i^s)^t
\end{multline*}
For $0\le t<1$, Jensen's inequality, Condition~\eqref{SN} and Lemma~\ref{pres-est} yield
\begin{multline*}
\sum_{n\ge N} ne^{-P_L(t) n}\sum_{\tau_i=n} \mu_\gamma(\gamma\cap\Lambda_i^s)^t \le \sum_{n\ge N} n e^{-P_L(t)n} S_n^{1-t}\left( \sum_{\tau_i=n} \mu_\gamma(\gamma\cap\Lambda_i^s) \right)^t \\
\le \left(\mu_{\gamma}(\gamma\cap\Lambda)\right)^t\sum_{n\ge N} n e^{n(1-t)(h+\int_x\varphi_1\,d\mu_1)}<K\theta^n
\end{multline*}
for some constants $K>0$ and $e^{(1-t)(h+\int_X\varphi_1\,d\mu_1)}<\theta<1$ which exist since $h<-\int_X\varphi_1\,d\mu_1$. For $t_0<t<0$ with $t_0$ from Lemma~\ref{interval}
\begin{multline*}
\sum_{n\ge N} ne^{-P_L(t) n}\sum_{\tau_i=n} \mu_\gamma(\gamma\cap\Lambda_i^s)^t
\le \sum_{n\ge N} ne^{-P_L(t) n} S_n  \lambda_1^{-nt} \\
\le \sum_{n\ge N} n\lambda_1^{-nt} e^{n(h+(1-t)\int_X\varphi_1\,d\mu_1)} <K\theta^n
\end{multline*}
for some constants $K>0$ and $\lambda_1^{-t} e^{h+(1-t)\int_X\varphi_1\,d\mu_1}<\theta<1$.
\end{proof}

\section{Applications I: Thermodynamics of the H\'enon family at the first birfurcation}\label{Henon}

As an application we outline the result in \cite{SenTak15} where the thermodynamical formalism presented in the previous sections is applied in order to establish the existence and uniqueness of equilibrium measures associated to the geometric potential for H\'enon-like diffeomorphisms of the plane at the first bifurcation parameter. For parameters $a, b\in\R$ and $(x,y)\in\R^2$ the H\'enon map $f_{a,b}$ is defined by 
$$
f_{a,b}(x,y):=(1-ax^2+\sqrt{b}y, \pm\sqrt{b}x).
$$
H\'enon-like maps are $C^2$-small perturbations of  $f_{a,b}$. It is proven in \cite{BedSmi04, BedSmi06, CaoLuzRio08} that for each 
$0<b\ll 1$ there exists a uniquely defined parameter $a^*=a^*(b)$ such that the non-wandering set for $f_{a,b}$ is a uniformly hyperbolic horseshoe for $a > a^*$ and the parameter $a^*$ is the first parameter value for which a homoclinic tangency  between certain stable and unstable manifolds appears. The results presented here apply to fixed parameter values $b$ sufficiently small and $a=a^*(b)$. The dependency on the parameters is thus usually dropped from the notation and $f:=f_{a^*,b}$.
The (one-dimensional) subspace denoted by $E^u(z)$ for which 
$$
\limsup_{n\to\infty}\frac{1}{n}\log\|Df^{-n}|_{E^u(z)}\|<0
$$  
is well-defined on the non-wandering set (see \cite[Proposition 4.1]{SenTak13}). Recall that $J^u(f)$ denotes the absolute value of the Jacobian along the unstable direction. The following result holds:
\begin{thm}\cite[Theorem A]{SenTak15}\label{Henon1}
For any bounded open interval $I\subset(-1,+\infty)$ there exists $0<b_0\ll 1$ such that if $0\le b<b_0$ then there exists a unique equilibrium measure associated to $-t\log J^u(f)$ for all $t\in I$.
\end{thm}
The theorem relies on the construction of an inducing scheme obtained through a  first return time function to a given region containing the tangency. This region must, however, be carefully chosen in order to ensure the control of the distortion. This is done as follows.

The map $f=f_{a*(b),b}$ has two hyperbolic fixed points, $Q$ close to $(-1, 0)$ and $P$ close to $(\frac12, 0)$, and the stable and unstable manifolds of $Q$ intersect at the point of homoclinic tangency $\zeta_0$. The largest rectangular region delimited by (two) pieces of stable and (two) pieces of unstable manifolds of $Q$ is denoted by $R$. (see {\sc Figure 1}).

\begin{figure}
\begin{center}
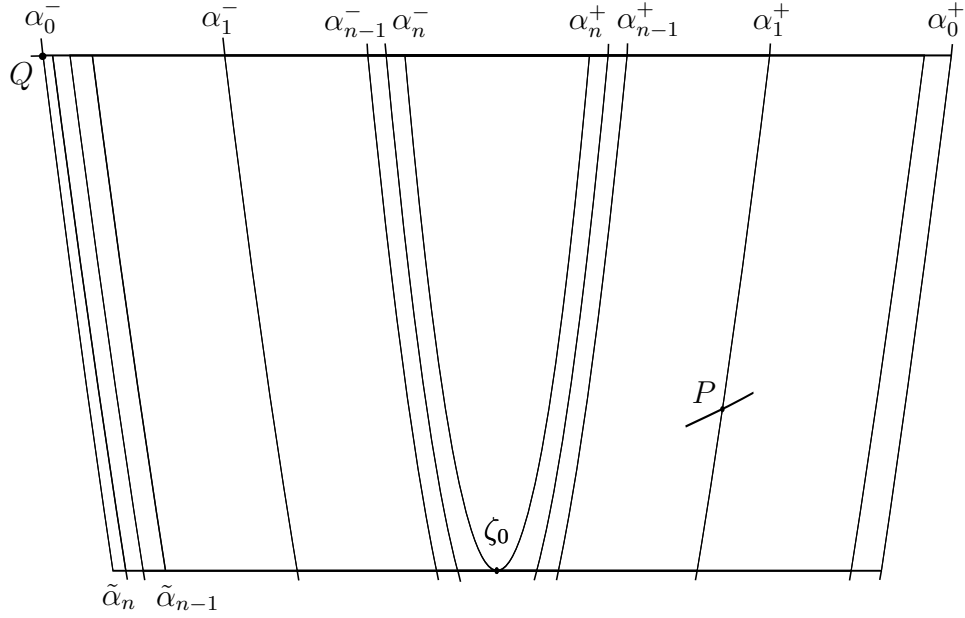
\caption{The rectangle $R$ and the curves $\{\tilde\alpha_n\}$, $\{\alpha_n^{+}\}$, $\{\alpha_n^{-}\}$.
The $\{\tilde\alpha_n\}$ accumulate on the left stable side of $R$. Both $\{\alpha_n^{+}\}$ and $\{\alpha_n^{-}\}$ accumulate on the parabola in $W^s(Q)$
containing the point of tangency $\zeta_0$.
near $(0,0)$.}
\end{center}
\end{figure}

Now define the family $\{\alpha_n^\pm\}_{n\ge0}$ of stable curves. First, let $\tilde\alpha_0$ be the component of $W^s(P)\cap R$ containing $P$ and define the sequence 
$\{\tilde\alpha_n\}_{n\ge0}$ of compact curves in $W^s(P)\cap R$ inductively: given $\tilde\alpha_{n-1}$, define $\tilde\alpha_n$ as the component of $f^{-1}(\tilde{\alpha}_{n-1})\cap R$ which lies to the left of $\zeta_0$. Then denote by $\alpha_{n+1}^-$ (respectively 
$\alpha_{n+1}^+$) the connected component of 
$f^{-2}\tilde\alpha_n\cap R$ which lies closest to $\zeta_0$ on the left (respectively on the right). Denote by $\Theta$ the rectangular region whose boundary consists of the curves $\tilde{\alpha}_0$ and $\tilde{\alpha}_1$ and the unstable boundary of $R$ and by 
$\Theta_n$ the region whose boundary consists of the curves 
$\alpha_{[\frac{10}{\varepsilon}]n+N}^\pm$ and the unstable boundary of $R$ for some given large $N$. Observe that $\Theta_n$ form a nested sequence of rectangles which accumulate to the component of 
$W^s(Q)\cap R$ containing $\zeta_0$. Then define the (closed) set 
$$
\Omega_\infty:=\overline{\{z\in\Theta\colon f^n(z)\in R\setminus\Theta_{n}\text{ for all } n\in\N\}},
$$
where $\overline{J}$ denotes the closure of a set $J$. An inducing scheme $\{S, \tau\}$ is then defined on the subset of those points in $\Omega_\infty$ which return to $\Omega_\infty$ infinitely often (in finite time). The inducing time is the first return time to $\Omega_\infty$ (see \cite[Section 3.5]{SenTak15} for details). One obtains the following: 

\begin{prop}\cite[Proposition 3.2]{SenTak15}\label{lat}
For any small $\varepsilon>0$ there exists $b_0>0$ such that if $0<b<b_0$, there exists a collection $S$ of pairwise disjoint Borel subsets and a function $\tau\colon S\to \mathbb N$ such that $\{S,\tau\}$ is an inducing scheme of hyperbolic type satisfying (I3) and (I4) for $f=f_{a^*(b)}$ and such that
$$
\varlimsup_{n\to\infty}\frac{1}{n}\log S_n\leq\varepsilon,
$$
where $S_n:=\#\{J\in  S\colon \tau(J)=n\}$. In particular, any measure 
$\mu\in\mathcal M(f,\R^2)$ with $h(\mu)\geq 2\varepsilon$ is liftable to the inducing scheme $\{S,\tau\}$.
\end{prop}

The proof that $\{S, \tau\}$ is an inducing scheme of hyperbolic type is carried out in \cite[Section 3.5]{SenTak15}. Note that (I4) follows from the fact that the fixed saddle $P$ belongs to one of the elements of the inducing domain $S$. Also, since the intersection of the closures of elements of $S$ are contained in the stable manifolds of $P$, one obtains a stronger Condition than (I3) (see Condition (A3) in \cite{SenTak15}). The rest of the proof of Proposition~\ref{lat} is carried out in Sections 3.5 and 3.6 of~\cite{SenTak15}.
\begin{prop}\cite[Corolary 4.8]{SenTak15}
For any bounded open interval $I\subset(-1,+\infty)$ there exists $0<b_1\ll 1$ such that if $0\le b<b_1$ then for all $t\in I$ the potential function $-t\log J^u(f)$ has strongly summable variations, finite Gurevic pressure and is positive recurrent. 
\end{prop}
There exist $-1<t_-<0<t_+$ (defined explicitely in \cite[Section 4.3]{SenTak15}) with $\lim_{b\to0}t_-=-1$ and $\lim_{b\to0}t_+=\infty$ and such that for any $t_-<t<t_+$ the induced potential $-t\log J^u(f)$ associated to the inducing scheme of Proposition~\ref{lat} has strongly summable variations and finite Gurevic pressure (see Lemmas 4.1 and 4.2 in \cite[Section 4.1]{SenTak15}) and is positive recurrent (see \cite[Lemm 4.7]{SenTak15}). 

Theorem~\ref{Henon1} now follows from Theorem~\ref{liftgibbs} after observing that for all bounded open $I\subset (-1,\infty)$ and 
$\varepsilon>0$ sufficiently small there exists  $b_1$ such that for all 
$0\le b<b_1$ one must have $I\subset(t_-,t_+)$ and the entropy of any equilibrium measure must be larger than $2\varepsilon$. By Proposition~\ref{lat}, such measures are liftable to the inducing scheme $\{S, \tau\}$  and Theorem~\ref{Henon1} follows.

\section{Applications II: Thermodynamics of the Katok map}\label{katok-map}
%----------------------------------------

We will outline another application of our results to effect the thermodynamics of the Katok map $f$ (see \cite{PesSenZha14} for a detailed presentation). This map was introduced in \cite{Kat79} and is an area preserving non-uniformly hyperbolic 
$C^\infty$ diffeomorphism of the $2$-torus. We shall show existence and uniqueness of equilibrium measures for the geometric potential $\varphi_t=-t\log \J f(x)$ for $0\le t\le 1$ and describe some of their ergodic properties.

\subsection{Definition of the Katok map.} Consider the automorphism of the two-dimensional torus $\T^2=\R^2/\Z^2$ given by the matrix
$T:=\left(\begin{smallmatrix} 2 & 1\\1 & 1\end{smallmatrix}\right)$ and then choose a function $\psi:[0,1]\mapsto[0,1]$ satisfying:
\begin{enumerate}
\item[(K1)] $\psi$ is of class $C^\infty$ except at zero;
\item[(K2)] $\psi(u)=1$ for $u\ge r_0$ and some $0<r_0<1$;
\item[(K3)] $\psi'(u)> 0$ for every $0<u<r_0$;
\item[(K4)] $\psi(u)=(ur_0)^\alpha$ for $0\le u\le\frac{r_0}{2}$ where $0<\alpha<\frac12$. 
\end{enumerate}
Let $D_r=\{(s_1,s_2): {s_1}^2+{s_2}^2\le r^2\}$
where $(s_1,s_2)$ is the coordinate system obtained from the eigendirections of $T$. Choose $r_1>r_0$ such that
\begin{equation}\label{numbers}
D_{r_0}\subset\Int T(D_{r_1})\cap \Int T^{-1}(D_{r_1})
\end{equation}
and consider the system of differential equations in $D_{r_1}$
\begin{equation}\label{batata10}
\dot{s}_1= s_1\log\lambda,\quad \dot{s}_2=-s_2\log\lambda,
\end{equation}
where $\lambda>1$ is the eigenvalue of $T$. Observe that $T$ is the time-one map of the flow generated by the system of equations \eqref{batata10}.

We slow down trajectories of the the system  \eqref{batata10} by perturbing it in $D_{r_1}$ as follows
\begin{equation}\label{batata2}
\begin{aligned}
\dot{s}_1=&\quad s_1\psi({s_1}^2+{s_2}^2)\log\lambda\\
\dot{s}_2=&- s_2\psi({s_1}^2+{s_2}^2)\log\lambda.
\end{aligned}
\end{equation}
This system of equations generates a local flow. Denote by $g$ the time-one map of this flow. The choices of $\psi$ and $r_0$ and $r_1$ guarantee that the domain of $g$ contains $D_{r_1}$.
Furthermore, $g$ is of class $C^\infty$ in $D_{r_1}$ except at the origin and it coincides with $T$ in some neighborhood of the boundary $\partial D_{r_1}$. Therefore, the map
\[
G(x)=\begin{cases} T(x) & \text{if $x\in\T^2\setminus D_{r_1}$,}\\
g(x) & \text{if $x\in D_{r_1}$}
\end{cases}
\]
defines a homeomorphism of the torus $\T^2$, which is a $C^\infty$
diffeomorphism everywhere except at the origin. 

The map $G$ preserves a probability measure $\nu$, which is absolutely continuous with respect to the area. The density of $\nu$ is a $C^\infty$ function 
that is infinite at~$0$. One can further perturb the map $G$ to obtain an area-preserving $C^\infty$ diffeomorphism $f$. 
This is the Katok map (see \cite{Kat79} and also \cite{BarPes13})
\footnote{Our construction of the Katok map differs from Katok's original construction in that the latter allows an arbitrary function $\psi$ which is $C^\infty$ everywhere but at the origin, is zero at the origin and satisfies 
$\int_0^1\frac{du}{\psi(u)}<\infty$.}.

We mention the following important properties of the map $f$.

\begin{prop}\label{kat1}
The map $f$ has nonzero Lyapunov exponents but at the origin and is topologically conjugate to the hyperbolic automorphism $T$ via a conjugacy homeomorphism $H$.
\end{prop}
\subsection{An inducing scheme for the Katok map.} We begin with an inducing scheme for the hyperbolic automorphism $T$. Consider a finite Markov partition $\Px$ for $T$ and let $P\in\Px$ be a partition element which does not contain the origin. Note that $P$ is a rectangle with hyperbolic product structure.
Moreover, given $\varepsilon>0$ we can always choose the Markov partition 
$\Px$ in such a way that $\text{diam }(P)<\varepsilon$ and 
$P=\overline{\Int P}$ for any $P\in\Px$. For a point $x\in P$ denote by $\gamma^s(x)$ (respectively, $\gamma^u(x)$) the connected component of the intersection of $P$ with the stable (respectively, unstable) leaf of $x$, which contains $x$. 

We construct an inducing scheme for $T$ as follows. Given 
$x\in P$, let $\tau(x)$ be the first return time of $x$ to $\Int P$.
For all $x$ with $\tau(x)<\infty$, the connected component of the level set $\tau=\tau(x)$ containing $x$ is 
$$
\Lambda^s(x)=\bigcup_{y\in U^u(x)\setminus A^u(x)}\,\gamma^s(y)
$$ 
where $U^u(x)\subseteq\gamma^u(x)$ is an interval containing $x$ and open in the induced topology of $\gamma^u(x)$, and 
$A^u(x)\subset U^u(x)$ is the set of points which either lie on the boundary of the Markov partition or  never return to the set $P$. Note that $A^u(x)$ has zero Lebesgue measure. By construction, 
$\Lambda^s(x)$ is an $s$-subset whose image under $T^{\tau(x)}$ is a $u$-subset containing $T^{\tau(x)}(x)$. Furthermore, for any 
$x, y\in P$ with finite first return time the sets $\Lambda^s(x)$ and 
$\Lambda^s(y)$ are either coincide or disjoint. 
\begin{prop}\label{kat2} The following statements hold:
\begin{enumerate}
\item The collection of sets $\{\Lambda_i^s\}$ and numbers 
$\{\tau_i\}$ generate an inducing scheme $\{S,\tau\}$ for $T$ which satisfies Conditions (I3) and (I4);
\item There exists $h_1<h_{top}(T)$ such that 
$$
S_n\le e^{h_1n};
$$
\end{enumerate}
\end{prop}
The first statement is straightforward.
For the second statement, it suffices to estimate the number of sets $\Lambda^s_i$ with a given $i$. This number is less than the number of periodic orbits of $T$ of minimal period $\tau_i$ that originate in $P$. Using the symbolic representation of $T$ as a subshift of finite type induced by the Markov partition $\mathcal{P}$, one can see that the latter equals the number of symbolic words of length $\tau_i$ for which the symbol $P$ occurs only as the first and last symbol (but nowhere in between). The number of such words grows exponentially with exponent $h_1 < h_{top}(T)$.

Using Proposition~\ref{kat2} and applying the conjugacy map $H$ from Proposition \ref{kat1}, one obtains an inducing scheme $\{H(S), \tau\}$ for the Katok map $f$, where $H(S)=\{\tilde J=H(J)\}_{J\in S}$, which satisfies Conditions (I3) and (I4). Since $H$ preserves topological and combinatorial information about $T$ (in particular, its topological entropy), the inducing schemes for $f$ and $T$ have the same number $S_n$ of basic elements with inducing time $\tau_i=n$.

On the other hand, one can obtain the same inducing scheme by representing the Katok map $f$ as a \emph{Young diffeomorphism} associated with the collection of $s$-subsets $H(\Lambda_i^s)$. 

We further restrict the choice of the partition element $P$. Given $Q>0$, we can take the number $r_0$ in the construction of the Katok map so small and, by refining the Markov partition if necessary, we can choose a partition element $P$ such that 
\begin{equation}\label{partition}
f^n(x)\notin D_{r_0} \text{ for any } 0\le n\le Q 
\end{equation}
and any point $x$ for which either $x\in P$ or $x\notin f(D_{r_0})$ while  $f^{-1}(x)\in D_{r_0}$.

\begin{prop}
There exists $Q>0$ such that for the chosen inducing scheme, the collection of $s$-subsets $H(\Lambda_i^s)$ satisfies Conditions (Y0)--(Y5). 
\end{prop}

\subsection{Equilibrium measures for the Katok map.} The following result describes existence, uniqueness and ergodic properties of the equilibrium measures associated to the geometric potential for the Katok map $f$.
\begin{thm}\label{Katok}
There exist $\varepsilon>0$ and $t_0<0$ such that for all $r_0<\varepsilon$ and every $t_0<t< 1$
there exists a unique ergodic equilibrium  measure $\mu_t$ associated to the geometric potential $\varphi_t=-t\log \J f$. 
This measure has exponential decay of correlations and satisfies the CLT with respect to a class of potential functions which includes all H\"older continuous functions on $\T^2$.
\end{thm}
\begin{proof}
First, observe that by statement (2) of Proposition~\ref{kat2} 
$h_1<h_{top}(T)=\log \lambda$. For sufficiently small $r_0>0$ and every $\varepsilon>0$ one can choose $r_1>0$ from~\eqref{numbers} such that 
$$
\left|\int\log \J f\,dm-\log\lambda\right|<\varepsilon.
$$ 
It follows that $h_1<-\int \varphi_1\,d\mu_1$ where $\mu_1=m$ is the area.
Theorems~\ref{geom_poten} and \ref{geom_poten1} and the fact that the inducing time is the first return time imply the existence of $t_0(P)<0$ such that for every $t_0(P)<t<1$ 
there exists a unique equilibrium measure $\mu_t$ associated to the geometric potential $\varphi_t$ among all measures $\mu$ for which $\mu(P)>0$. Note that 
$\mu_t(U)>0$ for every open set $U\subset P$. 

To prove uniqueness first observe that given another element $\tilde P$ of the Markov partition which satisfies Condition~\eqref{partition}, one can repeat the above argument: there hence exists $\tilde{t}_0=t_0(\tilde{P})$ such that for every $\tilde{t}_0<t<1$ 
there exists a unique equilibrium measure $\tilde{\mu}_t$ associated to the geometric potential among all measures $\mu$ for which 
$\mu(\tilde{P})>0$. Moreover $\tilde{\mu}_t(\tilde{U})>0$ for every open set $\tilde{U}\subset\tilde{P}$. Since the map $f$ is topologically transitive, for every open sets $U\subset P$ and 
$\tilde{U}\subset\tilde{P}$ there exists an integer $k$ such that $f^k(U)\cap\tilde{U}\ne\emptyset$. Therefore, 
$\mu_t=\tilde{\mu}_t$. Note that if the number $r_0$ in the construction of the Katok map is sufficiently small, the union of partition elements that satisfy Condition~\eqref{partition} form a closed
set whose complement is a neighborhood of zero. The desired result follows by observing that the only measure which does not charge any element of the Markov partition that lies outside this neighborhood is the Dirac measure $\delta_0$ at the origin whose pressure is $P(\delta_0)=0$ whereas $P(\mu_t)>0$ for every 
$$
t_0:=\max_{P\in\mathcal{P}, P\cap D_{r_1}=\emptyset}t_0(P)<t<1.
$$
The desired result follows. 
\end{proof}
\begin{rmk}
It follows from the entropy formula that the area is an equilibrium measure for $\varphi_1$. However there is another equilibrium measure that is the Dirac measure at zero. It is expected that the equilibrium measure $\mu_1$ (the area) has polynomial decay of correlations.
\end{rmk}

\bibliographystyle{alpha}
\bibliography{BiblioSenti2013}
 
\end{document}